%% file: main.tex
\documentclass[a4paper]{scrartcl}

\usepackage[utf8]{inputenc}
\usepackage{tikz}
\usetikzlibrary{trees,shapes}
\usepackage[appendix=append,bibliography=common]{apxproof}
\usepackage{relsize}
\usepackage{stackengine}
\stackMath
\usepackage{amsmath}
\usepackage{amssymb}
\usepackage{todonotes}
\usepackage{graphics,color,caption,subcaption}
\usepackage{hyperref}
\hypersetup{
	    colorlinks,
	    linkcolor={red!50!black},
	    citecolor={blue!50!black},
	    urlcolor={blue!30!black},
	    breaklinks=true
	    }
\usepackage{mathabx}
\usepackage{bbm}

\newtheoremrep{theorem}{Theorem}[section]

\newtheorem{conjecture}[theorem]{Conjecture}
\newtheorem{remark}[theorem]{Remark}
\newtheorem{fact}[theorem]{Fact}
\newtheoremrep{proposition}[theorem]{Proposition}
\newtheoremrep{claim}[theorem]{Claim}
\newtheoremrep{corollary}[theorem]{Corollary}
\newtheoremrep{lemma}[theorem]{Lemma}
\newtheoremrep{example}[theorem]{Example}
\newtheoremrep{definition}[theorem]{Definition}

\newcommand{\ZZ}{\mathbb{Z}}
\newcommand{\NN}{\mathbb{N}}

\input{macros}

\begin{document}
  \title{The Non-Cancelling Intersections Conjecture}
\author{Antoine Amarilli\\{\small LTCI, Télécom Paris,}\\[-.4em]{\small Institut Polytechnique de
Paris}\\{\small
\texttt{a3nm@a3nm.net}} \and Mikaël Monet\\{\small
Université de Lille, CNRS, Inria, UMR 9189 - CRIStAL, F-59000 Lille, France}\\{\small \texttt{mikael.monet@inria.fr}} \and
Dan Suciu \\{\small University of Washington}\\{\small
\texttt{suciu@cs.washington.edu}}}
\date{}
\maketitle

\begin{abstract}
  In this note, we present a conjecture on intersections of set
  families, and a rephrasing of the conjecture in terms of principal downsets
  of Boolean lattices. 
  The conjecture informally states that, whenever we can
  express the measure of a union of sets in terms of the
  measure of some of their intersections using the inclusion-exclusion formula,
  then we can express the union as a set from these same intersections via the set operations
  of disjoint union and subset complement. We also present a partial result
  towards establishing the conjecture.
\end{abstract}

\paragraph*{Acknowledgements.}
We thank Louis Jachiet for helpful discussions on the problem, and for having found
the counterexample to a stronger version of the conjecture that is mentioned at the end
of Section~\ref{sec:final}.
This work was done in part while the authors were visiting the Simons Institute
for the Theory of Computing.

\section{Introduction}
\label{sec:intro}
\input{introduction}

\section{Preliminaries}
\label{sec:prelims}
\input{preliminaries}

\section{The Non-Cancelling Intersections Conjecture}
\label{sec:conj}
\begin{toappendix}
\label{apx:conj}
\end{toappendix}
\input{conj.tex}

\section{Simplifying to Tight Intersection Lattices}
\label{sec:tight}
\begin{toappendix}
\label{apx:tight}
\end{toappendix}
\input{tight.tex}

\section{Alternative Formulation: Dot-Algebra of Non-Cancelling Principal
Downsets in the Boolean Lattice}
\label{sec:alt}
\begin{toappendix}
\label{apx:alt}
\end{toappendix}
\input{alt}

\section{Partial result on $\ncpd$: avoiding a given zero}
\label{sec:partial}
\input{partial}

\section{Extensions, Variants and Counterexample Search}
\label{sec:final}
\input{final}

\bibliographystyle{apalike}
\bibliography{main}

\newpage

\appendix

\end{document}

%% file: macros.tex
\usepackage{qtree}

\newcommand{\RR}{\mathbb{R}}
\newcommand\eul{\mathsf{e}}
\newcommand\bbmB{\mathbbm{B}}
\newcommand\bbmL{\mathbbm{L}}
\newcommand\bbmU{\mathbbm{U}}

\newcommand\frakF{\uparrow}
\newcommand\frakI{\downarrow}
\newcommand\calF{\mathcal{F}}
\newcommand\calC{\mathcal{C}}
\newcommand\calP{\mathcal{P}}
\newcommand\calI{\mathcal{I}}
\newcommand\calT{\mathcal{T}}
\newcommand\calA{\mathcal{A}}
\newcommand\calG{\mathcal{G}}

\newcommand\nci{\mathtt{NCI}}
\newcommand\ncu{\mathtt{NCU}}
\newcommand\nti{\mathtt{NTI}}
\newcommand\ncpd{\mathtt{NCPD}}
\newcommand\ntz{\mathtt{NTZ}}
\newcommand\ntcz{\mathtt{NTCZ}}

\newcommand\pd{\mathtt{PD}}
\newcommand\ap{\mathtt{AP}}
\newcommand\pol{\mathtt{pol}}
\newcommand\mult{\mathtt{mult}}
\newcommand\lift{\mathtt{lift}}
\newcommand\onehat{\hat{1}}
\newcommand\zerohat{\hat{0}}
\newcommand\muhat{\hat{\mu}}
\newcommand\mucheck{\check{\mu}}

\newcommand{\dual}{\mathrm{dual}}

\def\smallbullet{\mbox{\larger[-5]$\bullet$}}
\def\cupdot {\stackrel{\smallbullet}{\cup}}

\def\minusdot {\stackrel{\smallbullet}{\setminus}}

\newcommand\join{\mathrel{\text{\raisebox{0.25ex}{\scalebox{0.8}{$\vee$}}}}}
\newcommand\meet{\mathrel{\text{\raisebox{0.25ex}{\scalebox{0.8}{$\wedge$}}}}}

\newcommand\defeq{\stackrel{\mathrm{def}}{=}}

\newcommand{\squig}{{\scriptstyle\sim\mkern-3.9mu}}

\newcommand{\rsquigend}{{\scriptstyle\rule{.1ex}{0ex}\rhd}}
\newcounter{sqindex}
\newcommand\squigs[1]{%
  \setcounter{sqindex}{0}%
  \whiledo {\value{sqindex}< #1}{\addtocounter{sqindex}{1}\squig}%
}
\newcommand\rewr[2]{%
  \mathbin{\stackon[2pt]{\squigs{#2}\rsquigend}{\scriptscriptstyle\text{#1\,}}}%
}

%% file: introduction.tex
We present in this note a conjecture on intersection lattices of set
families, which can be equivalently stated on the Boolean lattice.
The original motivation for the conjecture comes from a problem in database
theory about the existence of certain circuit representations for
probabilistic query evaluation (see~\cite{monet2020solving}), but in this note
we present the conjecture as a purely abstract claim without any database
prerequisites. A positive answer to this abstract conjecture implies a positive answer to
the database theory conjecture.

The conjecture can be understood in terms of the inclusion-exclusion formula.
Consider a family of sets $S_1, \ldots, S_n$. The inclusion-exclusion formula
can be used to express a quantity on the union $S_1 \cup \cdots \cup S_n$ (e.g.,
the cardinality, or the value of some measure) as a function of
the intersections of the $S_i$. 
In general, some of these intersections may in fact be identical: we can in
fact define the \emph{intersection lattice} of the set family to represent the
possible intersections that can be obtained. Further, we can
have \emph{cancellations}, i.e., some of the possible intersections may end up
having a coefficient of zero in the inclusion-exclusion formula, due to
cancellations. Thus, in the general case, inclusion-exclusion allows us to
express the measure of $S_1 \cup \cdots \cup S_n$ in terms of the measure of the
intersections $I_1, \ldots, I_m$ that have a non-zero coefficient.

Our conjecture asks whether, in this case, one can obtain the \emph{set} $S_1 \cup
\cdots \cup S_n$ from the intersections $I_1, \ldots, I_m$, using the set
operations of disjoint union and subset-complement. If this is true, then it
implies in particular the inclusion-exclusion formulation, provided that our
measure is additive in the sense that the measure of $S \cupdot S'$ (where~$\cupdot$ denotes disjoint union and
~$\minusdot$ subset complement, as per Definition~\ref{def:dots}) is the sum
of the measure of $S$ and of~$S'$. Obtaining
such a set expression can sometimes be done simply by re-ordering the
inclusion-expression formula: for instance, if we write $|X\cup Y| = |X| + |Y| -
|X \cap Y|$, we can reorder to $|X\cup Y| = |X| -
|X \cap Y| + |Y|$, and then express $X \cup Y = (X \minusdot (X \cap Y)) \cupdot
Y$. However, in general, it is unclear whether such an expression can be
obtained. 
Before continuing, let us give a first toy example.

\begin{example}

Consider the set family $\calF = \{S_1,S_2,S_3\}$ with  $S_1=\{a,b,d\}$, $S_2=\{a,b,c,e\}$,
$S_3=\{a,c,f\}$. Then:
\begin{align*}
  |S_1 \cup S_2 \cup S_3| &= |S_1| + |S_2| + |S_3| \\
                          &\phantom{=} - (|S_1 \cap S_2| + |S_1 \cap S_3| + |S_2 \cap S_3|) \\
                          &\phantom{=} + |S_1 \cap S_2 \cap S_3|.
\end{align*}
Since we have~$S_1 \cap S_3 = S_1 \cap S_2 \cap S_3$, we obtain
\begin{align*}
  |S_1 \cup S_2 \cup S_3| &= |\{a,b,d\}| + |\{a,b,c,e\}| + |\{a,c,f\}|- |\{a,b\}| - |\{a,c\}|.
\end{align*}
The non-cancelling intersections are the ones that remain, i.e., $\{a,b,d\}$,
$\{a,b,c,e\}$, $\{a,c,f\}$, $\{a,b\}$, and $\{a,c\}$, while $\{a\}$ ($=S_1 \cap
S_3 = S_1 \cap S_2 \cap S_3$) is a cancelling term.

Then, we can express~$S_1 \cup S_2 \cup S_3 = \{a,b,c,d,e,f\}$ with $
\big[\big((\{a,b,d\} \minusdot \{a,b\}) \cupdot \{a,b,c,e\}\big) \minusdot
\{a,c\}\big] \cupdot \{a,c,f\}$: the reader can easily check that each
$\cupdot$ (resp., each $\minusdot$) is a valid disjoint union (resp., subset
complement), and that we have only used the non-cancelling intersections. Note
that this is not the only valid expression, for instance we can also obtain the
union with the expression $[\{a,b,d\} \minusdot \{a,b\}] \cup [\{a,b,c,e\}
\minusdot \{a,c\}] \cup \{a,c,f\}$.

\end{example}

Our results imply that it is always possible when there are no cancellations
(Fact~\ref{fact:annoyingfact}), or exactly one cancellation
(in the equivalent formulation on Boolean lattices, Theorem~\ref{thm:partial}), but in general it seem challenging to do so while
avoiding those intersections which cancel out.
The goal of this note is to present the current status of our efforts in
attacking the conjecture, in particular showcasing some equivalent formulations,
presenting examples, and establishing a very partial result. After presenting
preliminaries in Section~\ref{sec:prelims}, we formally state the conjecture and
give examples in Section~\ref{sec:conj}. We show in Section~\ref{sec:tight} that
it suffices to study the conjecture on specific intersection lattices, which we
call \emph{tight}, where informally the set family ensures that every possible
intersection contains exactly one element that is ``specific'' to this
intersection (in the sense that it is present precisely in this intersection and
in larger intersections). We use this in Section~\ref{sec:alt} to give an
alternative formulation: instead of working with intersection lattices, we can
work with downsets on the Boolean lattice. We show that this is equivalent to
the original conjecture. The hope is that the setting of the Boolean lattice can
be more convenient to work with, as its structure is more restricted, and we can
define quantities such as the Euler characteristic that seem useful in
understanding the structure of the problem.

We next present in Section~\ref{sec:partial} a partial result that we can show
in the context of the Boolean lattice. In the rephrased problem, we must express
a downset of that lattice using disjoint union and subset complement on
principal downsets spanned by elements which do not ``cancel out'' in the sense
of having non-zero Möbius value. Our result establishes that the rephrased conjecture is true
when there is one single node of the downset that has such value; this non-trivially 
extends the fact that the result is true when no such nodes exist
(Lemma~\ref{lem:allreach}), but falls short of the goal as 
in general many such zeroes can occur.

We conclude in Section~\ref{sec:final} with further questions and
directions on the conjecture. We point out that the conjecture can be
strengthened, in two different ways, to restrict the shape of the disjoint-union
and subset-complement expressions: we also do not know the status of these
stronger conjectures. We briefly report on an unsuccessful experimental search
for counterexamples. We also hint at an incomplete proof of another partial
result where we can avoid  more than one zero, provided the targeted downset
satisfies a certain decomposability condition.

%% file: preliminaries.tex
For a set~$S$ we write~$2^S$ its powerset.
In this work, by \emph{family of sets}, or simply \emph{set family}, we always mean
a finite set of (not necessarily finite) sets. We generally use
cursive letters to denote set families, uppercase letters to denote sets, and lowercase letters for elements of sets.
For a set family~$\calF$ (resp., non-empty set family~$\calF$) we write~$\bigcup \calF$ (resp., $\bigcap \calF$) for~$\bigcup_{X\in \calF} X$ (resp., $\bigcap_{X\in \calF} X$).
For a set~$X$ and two functions~$f,g\colon X \to \ZZ$, we
write~$f+g$ (resp., $f-g$) the function defined by~$(f+g)(x) \defeq f(x)+g(x)$
(resp., $(f-g)(x) \defeq f(x)-g(x)$) for all~$x\in X$.

\paragraph*{Posets.}
Recall that a \emph{poset}~$P=(A,\leq)$ is a pair consisting of a set~$A$ and a
binary \emph{partial order} relation~$\leq$ over~$A$ that is reflexive,
antisymmetric and transitive. 
By slight abuse of notation, we may write~$x\in
P$ to mean~$x\in A$ (and~$U\subseteq P$ to mean~$U\subseteq A$). 

A \emph{greatest} (resp., \emph{least}) element of~$P$ is an
element~$x\in P$ such that for all~$x'\in P$ we have $x' \leq x$
(resp., $x \leq x'$). When such an element exists it is unique, and
we then denote it by~$\onehat$ (resp.,~$\zerohat$). We may write~$\zerohat_P, \onehat_P$ to avoid confusion when multiple posets are involved.

For~$G \subseteq A$, we define the \emph{upset of~$P$ generated by~$G$} (also called an \emph{order
filter}),
denoted~$\frakF_{P}(G)$, by $\frakF_P(G) \defeq \{x\in P \mid \exists y\in G
\text{~s.t.~} y \leq x\}$.  We
also define the \emph{downset generated by~$G$} (also called an \emph{order
ideal}), denoted~$\frakI_P(G)$, by
$\frakI_P(G) \defeq \{x \in P \mid \exists y\in G \text{~s.t.~} x \leq y\}$.
Note that~$\frakF_P(\emptyset) = \frakI_P(\emptyset) = \emptyset$.
An \emph{upset} (resp., a
\emph{downset}) of~$P$ is a subset of~$P$ of the form~$\frakF_P(G)$
(resp.,~$\frakI_P(G)$) for~$G\subseteq A$. When~$|G|=1$ we call~$\frakF_P(G)$
(resp., $\frakI_P(G)$) a \emph{principal upset} (resp., \emph{principal
downset}), and by slight abuse of notation we sometimes write~$\frakF_P(x)$
for~$x\in P$ to mean~$\frakF_P(\{x\})$ (resp.,~$\frakI_P(x)$ to
mean~$\frakI_P(\{x\})$).

Two posets~$P_1 = (A_1, \leq_1)$ and~$P_2 = (A_2, \leq_2)$ are
\emph{isomorphic} if there exists a bijection~$h\colon A_1\to A_2$ such that for
all~$x,y\in P_1$ we have~$x \leq_1 y$ iff $h(x) \leq_2 h(y)$. When this is the case we write~$P_1\simeq P_2$.

\paragraph*{Lattices.}
A poset~$P=(A,\leq)$ is a \emph{meet semilattice} if, for every~$x,y\in P$,
there exists~$z\in P$ such that: (1) $z \leq x$ and~$z\leq y$; and (2) for
every~$z'\in P$ such that $z' \leq x$ and~$z'\leq y$, we have~$z' \leq z$. This
element~$z$ is then unique and is called the \emph{meet of~$x$ and~$y$}, and
written~$x \meet y$. We similarly define \emph{join semilattices} in the
expected way, with the join denoted~$x\join y$, and then \emph{lattices} as
posets that are both a meet semilattice and a join semilattice.

We recall that a finite join (resp., meet) semilattice that has a least
(resp., greatest) element is a lattice (see, e.g., \cite[Proposition
3.3.1]{stanley2011enumerative}).
We also recall the property that in a meet semilattice~$L$, the intersection of two
principal downsets~$\frakI_L(x)$ and~$\frakI_L(y)$ generated by~$x,y\in L$ is the
principal downset generated by~$x\meet y$; the dual property holds about principal upsets in
join semilattices with the join operation.

For a set~$S$ we write~$\bbmB_S$ the Boolean lattice over~$S$, by which we mean
the lattice~\mbox{$(2^S, \subseteq)$}, where join corresponds to set union and meet to
set intersection.

\begin{definition}
  \label{def:nontrivial}
  A set family~$\calF$ is called \emph{trivial} if
  it is empty or if
  there is~$X\in \calF$ such that~$\bigcup \calF = X$, and \emph{non-trivial} otherwise.
\end{definition}

We define the standard notion of \emph{intersection lattice} of a non-trivial set family:

\begin{definition}[\null\cite{stanley2011enumerative}, Section~3.7.2]
  \label{def:intersection-lattice}
  Let~$\calF$ be a non-trivial set family. For~$\calT \subseteq \calF$, $\calT \neq
  \emptyset$, we define:
  $S_{\calT} \defeq \bigcap \calT$. We also define $S_\emptyset \defeq
  \bigcup \calF$ with a slight abuse of notation as $S_\emptyset$ depends on
  the underlying set family~$\calF$.

  Define the poset~$\bbmL_\calF = (\calA,\subseteq)$, where $\calA \defeq
  \{S_{\calT} \mid \calT \subseteq \calF\}$.
  Then one can check
  that~$\bbmL_\calF$ is a meet semilattice whose meet operation is set intersection.
  We
  call~$\bbmL_\calF$ \emph{the intersection lattice} of~$\calF$. We say that a
  lattice~$L$ is \emph{an intersection lattice} if it is of the form $\bbmL_\calF$ for some non-trivial set family~$\calF$.
  (Note that we might have~$\bbmL_{\calF_1} = \bbmL_{\calF_2}$ for~$\calF_1 \neq
  \calF_2$.)
\end{definition}

\begin{remark}
  \label{rmk:lattices-are-intersection}
  Every finite lattice $L = (A,\leq)$ is isomorphic to some intersection
  lattice. Indeed, define $\calF \defeq \{\downarrow_L(\{U\}) \mid U \in A, U \neq \onehat_L\}$, i.e., the set of principal
  downsets of $L$ except $\downarrow_L(\{\onehat_L\})$ (as otherwise $\calF$ is
  trivial). Then one can easily check that $L \simeq \bbmL_\calF$, since
  $\downarrow_L(\{U\}) \cap \downarrow_L(\{V\}) = \downarrow_L(U \wedge_L V)$
  for $U,V\in A$.
\end{remark}

Such a lattice $L=\bbmL_\calF$ has a~$\zerohat$, which is~$S_{\calF} = \bigcap \calF$, and
a~$\onehat$, which is~$S_\emptyset = \bigcup \calF$. 
Notice that~$\onehat \neq \zerohat$ because $\calF$ is non-trivial.
Since~$L$ is a finite meet-semilattice that
has a~$\onehat$, we know that~$L$ is also a join-semilattice by the property mentioned
above, but for the purpose of this note we will not need to know what the join operation corresponds to.
Note that by definition $S_{\calT_1 \cup \calT_2} = S_{\calT_1} \cap S_{\calT_2}$ for
any~$\calT_1,\calT_2 \subseteq \calF$. Observe that we may have $S_{\calT_1} = S_{\calT_2}$ with $\calT_1 \neq
\calT_2$, in which case $S_{\calT_1} = S_{\calT_2} = S_{\calT_1 \cup \calT_2}$.

  \begin{example}
    \label{expl:intersection-lattices}
    Figure~\ref{fig:int-latt} shows the Hasse diagrams of the various intersection
    lattices defined next. (Ignore for now the integer annotations next to
    the nodes and the fact that some nodes are colored.)
    \begin{description}
      \item[\textbf{$L_1$.}]
  The intersection lattice~$L_1$ of (the non-trivial set family containing) $S_1 = \{n \in \NN \mid n \equiv 0 \pmod
  2\}$, $S_2 = \{n \in \NN \mid n \equiv 0 \pmod 3\}$, and $S_3 = \{n\in \NN
  \mid n \equiv 0 \pmod {12}\}$ is shown in Figure~\ref{fig:L1}. 

      \item[\textbf{$L_2$.}]
  The intersection lattice~$L_2$ of~$S_1 = \{a,d\}$, $S_2 = \{b,d\}$, and $S_3
  = \{c,d\}$ is shown in Figure~\ref{fig:L2}. 
        This example illustrates that an intersection lattice is not necessarily
  distributive.

      \item[\textbf{$L_3$.}]
  The intersection lattice~$L_3$ of~$S_1 = \{a,b\}$, $S_2 = \{a,c\}$, $S_3 =
  \{b,c\}$, and $S_4 = \{d\}$ is shown in Figure~\ref{fig:L3}. 

      \item[\textbf{$L_4$.}]
  The intersection lattice~$L_4$ of~$S_1 = \{a,c,d,g\}$, $S_2 = \{a,b,d,f\}$,
  $S_3 = \{a,b,c,e\}$, and $S_4 = \{a,h\}$ is shown in Figure~\ref{fig:L4}.
  Observe that~$L_3$ and~$L_4$ are isomorphic.
    \end{description}
    Finally, another intersection lattice~$L_5$ is depicted in Figure~\ref{fig:complex} (ignore for now the tree).
  \end{example}

\begin{figure}
  \begin{minipage}{\linewidth}

    \begin{minipage}[t]{.5\linewidth}  
    \noindent  
    \begin{subfigure}[b]{\linewidth}
      \centering
      \scalebox{0.65}{
	\begin{tikzpicture}
	\tikzset{nodestyle/.style={draw,rectangle}}
	\input{figures/L1}
    \end{tikzpicture}}
        \caption{Lattice $L_1$}
        \label{fig:L1}
      \end{subfigure}
      \end{minipage}
      \hfill
      \noindent
    \begin{minipage}[t]{.5\linewidth}  
    \noindent  
    \begin{subfigure}[b]{\linewidth}
      \centering
      \raisebox{.5cm}{
      \scalebox{0.65}{
	\begin{tikzpicture}
	\tikzset{nodestyle/.style={draw,rectangle}}
	\input{figures/L2}
  \end{tikzpicture}}}
        \caption{Lattice $L_2$}
        \label{fig:L2}
      \end{subfigure}
      \end{minipage}

    \end{minipage}\vspace{.3cm}

  \begin{minipage}{\linewidth}
    \begin{minipage}[t]{.5\linewidth}  
    \noindent  
    \begin{subfigure}[b]{\linewidth}
      \centering
      \scalebox{0.65}{
	\begin{tikzpicture}
	\tikzset{nodestyle/.style={draw,rectangle}}
	\input{figures/L3}
    \end{tikzpicture}}
        \caption{Lattice $L_3$}
        \label{fig:L3}
      \end{subfigure}
      \end{minipage}
      \hfill
\noindent
    \begin{minipage}[t]{.5\linewidth}  
    \begin{subfigure}[b]{\linewidth}
      \centering
      \scalebox{0.65}{
	\begin{tikzpicture}
	\tikzset{nodestyle/.style={draw,rectangle}}
	\input{figures/L4}
    \end{tikzpicture}}
        \caption{Lattice $L_4$}
        \label{fig:L4}
      \end{subfigure}
      \end{minipage}

  \end{minipage}
  \caption{Hasse diagrams of the intersection lattices from
    Example~\ref{expl:intersection-lattices}. The integer value besides each
  node~$n$ is~$\mu_L(n,\onehat)$ and is computed top-down following Definition~\ref{def:mobius}.
  The orange nodes are the non-cancelling non-trivial intersections.}
  \label{fig:int-latt}
\end{figure}
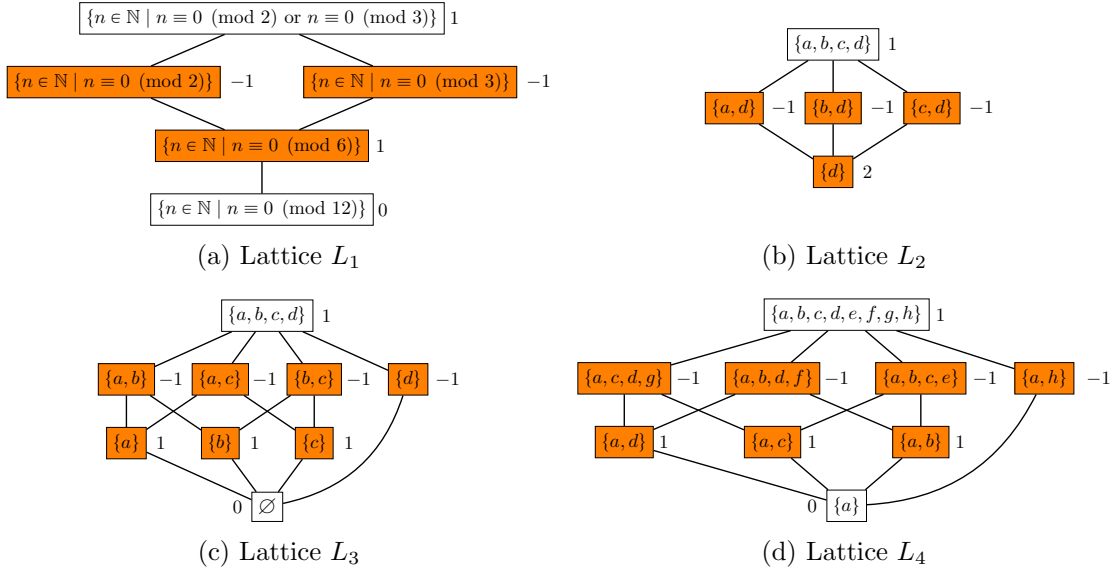

\begin{figure}
  \begin{minipage}{\linewidth}
  \begin{tikzpicture}[xscale=2.3,yscale=1.5]
	\tikzset{nodestyle/.style={draw,rectangle}}
\input{figures/Lcomplex}
\end{tikzpicture}
\end{minipage}\vspace{1cm}
  \begin{minipage}{\linewidth}
        \null\hfill
        \scalebox{0.8}{
  	\begin{tikzpicture}[level distance=1.2cm,
    level all/.style={sibling distance=3cm}]
          \usetikzlibrary{trees}
  	\input{figures/Lcomplextree}
      \end{tikzpicture}}
    \hfill\null
  \end{minipage}

  \caption{A tight intersection lattice of seven sets, and a witnessing tree
  showing that it does not violate the conjecture. For brevity we omit curly
  braces and commas when writing sets, i.e., $ab$ stands for $\{a,
  b\}$. }
  \label{fig:complex}
\end{figure}

In an intersection lattice~$L = \bbmL_\calF$, for
every~$x\in \onehat$, the set of elements of~$L$ in which~$x$ occurs
is the principal upset of~$L$ generated by
$\min_L(x) \defeq S_{\{X\in \calF \mid x\in X\}}$.    Note that~$\min_L(x) \neq \onehat$
because $\calF$ is non-trivial. In particular we have the following:

\begin{fact}
  \label{fact:union-min}
  Let~$L$ be an intersection lattice and~$U\in L$. Then we have~$U
  = \bigcup_{\substack{U'\in L \\ U' \subseteq U}} \{x\in \onehat \mid U' = \min_L(x)\}$, and this union is disjoint.
\end{fact}

\paragraph*{Möbius inversion formula and inclusion-exclusion.}
We now introduce the \emph{Möbius function} on finite posets.

\begin{definition}
  \label{def:mobius}
  Let~$P = (A,\leq)$ be a finite poset. The \emph{Möbius function}
  of~$P$ is the function~$\mu_P$ with value in~$\ZZ$ and domain the set of pairs~$(x,y)\in A \times A$ such that~$x\leq y$, defined,
  for each $y \in A$, by induction over the elements $x$ such that~$x\leq y$:
  \begin{itemize}
    \item $\mu_P(y,y) \defeq 1$;
    \item $\mu_P(x,y) \defeq - \sum_{x < z \leq y} \mu_P(z,y)$.\footnote{The Möbius
      function is traditionally defined with~$\mu_P(x,y) \defeq - \sum_{x \leq z <
    y} \mu_P(x,z)$, but the two definitions are
  equivalent~\cite[Proposition 3]{rota1964foundations}. We
use this version as it simplifies our presentation.}
  \end{itemize}
  \end{definition}

  \begin{proposition}[Möbius inversion formula; see, e.g., Prop.\ 3.7.1 of~\cite{stanley2011enumerative}]
    \label{prp:mob-inversion}
   Let~$P=(A,\leq)$ be a finite poset,
   and let~$f,g\colon A \to \RR$. Then we have:
    \[g(x) =
    \sum_{y\leq x} f(y) \text{~for all~} x\in P \text{~~~~if and only if~~~~} f(x) = \sum_{y\leq x}
    \mu_P(y,x) g(y) \text{~for all~} x\in P.
    \]
  \end{proposition}

  One important application of the Möbius inversion formula is in
  determining the \emph{measure} of a union of sets from the measure
  of some of their intersections.  For our purposes, given a set~$S$,
  what we call a \emph{measure on~$2^S$} is a
  function~$\xi\colon 2^S\to \RR$ which is nonnegative (i.e.,
  $\xi(U)\geq 0$ for all~$U\in 2^S$), which satisfies
  $\xi(\emptyset)=0$, and which is additive under countable disjoint
  unions, i.e., for any countable collection $\{E_k\}_{k=1}^\infty$ of
  pairwise disjoint subsets of~$S$, we have
  $\xi\left(\bigcup_{k=1}^\infty E_k\right) = \sum_{k=1}^\infty
  \xi(E_k)$. In particular, $\xi$ is completely determined by
  the images of the singleton sets $\{x\}$ for $x \in S$.

\begin{proposition}
  \label{prp:incl-excl-mob}
  Let $L$ be an intersection lattice,
  and let~$\xi$ be a measure
  on~$2^{\onehat}$. Then we have
  \begin{equation}\label{eq:incl-excl-mob}
  \xi(\onehat) = - \sum_{\substack{U\in L \\ U \neq \onehat}}  \mu_L(U,\onehat) \xi(U).
  \end{equation}
\end{proposition}
\begin{proof}
  This is a known result but, to be self-contained, we reproduce the proof from
  \cite{stanley2011enumerative} (end of Section 3.7, page~304). 
  Let us define two functions $f,g\colon L\to\RR$.
  For~$U \in L$, let~$g(U)
  \defeq \xi(U)$, and define~$f(U)$ to be the measure of all elements of~$U$
  which belong to no~$U' \subsetneq U$, i.e., $f(U)
  \defeq \xi(\{x \in \onehat \mid \min_L(x) = U \})$.
  Observe that~$g(U) = \sum_{\substack{U' \in L\\ U' \leq U}} f(U')$ for
  all~$U \in L$, by additivity of~$\xi$ and using Fact~\ref{fact:union-min}.
  Therefore by
  Proposition~\ref{prp:mob-inversion} we have that~$f(\onehat) = \sum_{U \in L}
  \mu_L(U, \onehat) g(U)$. But notice that~$g(\onehat) = \xi(\onehat)$ and
  that~$f(\onehat) = \xi(\emptyset)$ (because $L=\bbmL_\calF$ for some~$\calF$ non-trivial), so that~$f(\onehat) = 0$. Therefore 
  indeed~$\xi(\onehat) = - \sum_{\substack{U\in L \\ U \neq \onehat}} \mu_L(U,\onehat) \xi(U)$.
\end{proof}

We point out that the above proposition can fail if the underlying
set family~$\calF$ is trivial\footnote{Take for instance~$\calF =
\{ \{a,b\},\{a\}\}$.}: this is the reason why we always 
work with non-trivial families of sets.

\begin{example}
  \label{expl:mobius-of-intersection-lattices}
  For the lattices~$L_1,L_2,L_3,L_4,L_5$ from
  Example~\ref{expl:intersection-lattices}, the
  values~$\mu_{L_i}(U,\onehat)$ for all elements~$U\in L_i$ are
  shown next to the corresponding nodes in
  Figures~\ref{fig:int-latt} and~\ref{fig:complex}. 
\end{example}

This motivates the following definitions.

\begin{definition}
  \label{def:nti-nci}
  Let~$L$ be an intersection lattice. We call~$\onehat$ the \emph{trivial intersection}. 
We define the \emph{non-trivial intersections of~$L$}, denoted~$\nti(L)$, by
   \[\nti(L) \defeq \{U \in L \mid U \neq \onehat\},\]
  and we define the \emph{non-cancelling,
  non-trivial intersections} (or simply \emph{non-cancelling intersections})
  of~$L$, denoted~$\nci(L)$, by
   \[\nci(L) \defeq \{U \in L \mid U \neq \onehat \text{ and } \mu_L(U,\onehat) \neq 0\}.\]
\end{definition}

In other words, the non-cancelling intersections are those
intersections that do not cancel in the inclusion-exclusion
formula, i.e., those $U$ such that~$\xi(U)$ occur in
Equation~\eqref{eq:incl-excl-mob} with non-zero coefficient.

Last, we will also need the following application of the Möbius inversion
formula, which is another formulation of the inclusion-exclusion
principle.

\begin{proposition}
\label{prp:inv_bool}
Let~$S$ be a finite set, and let~$f,g\colon \bbmB_S \to \mathbb{R}$. Then we have
    $g(X) = \sum_{\substack{X'\in \bbmB_S \\ X\subseteq X'}} f(X')$
    for all $X\in \bbmB_S$ if and only if $f(X) =
    \sum_{\substack{X'\in \bbmB_S \\ X\subseteq X'}}
    (-1)^{|X'|-|X|} g(X')$ for all $X\in \bbmB_S$.
\end{proposition}

This result can be obtained by observing that 
$\mu_{\bbmB_S}(X,Y) = (-1)^{|Y|-|X|}$ for~$X\subseteq Y$ as shown in~\cite[Example
3.8.1]{stanley2011enumerative} and using the dual form of the Möbius inversion
formula~\cite[Proposition 3.7.2]{stanley2011enumerative}.

%% file: figures/L1.tex
 ===== NODES ====

 \node[nodestyle] (cup) at (0.0, 3.9) {$\{n\in \NN \mid n \equiv 0 \pmod 2 \text{ or } n \equiv 0 \pmod 3\}$};

\node[nodestyle,fill=orange] (2) at (-3, 2.6) {$\{n\in \NN \mid n \equiv 0 \pmod 2\}$};
\node[nodestyle,fill=orange] (3) at (3, 2.6) {$\{n\in \NN \mid n \equiv 0 \pmod 3\}$};
\node[nodestyle,fill=orange] (6) at (0, 1.3) {$\{n\in \NN \mid n \equiv 0 \pmod 6\}$};
\node[nodestyle] (12) at (0, 0) {$\{n\in \NN \mid n \equiv 0 \pmod {12}\}$};

 ===== MOBIUS ====

\node[right of=cup,xshift=+2.9cm] {$1$};
\node[right of=2,xshift=+1.6cm] {$-1$};
\node[right of=3,xshift=+1.6cm] {$-1$};
\node[right of=6,xshift=+1.4cm] {$1$};
\node[right of=12,xshift=+1.45cm] {$0$};

 ===== EDGES ====

\draw[black,thick] (cup) -- (2);
\draw[black,thick] (cup) -- (3);
\draw[black,thick] (2) -- (6);
\draw[black,thick] (3) -- (6);
\draw[black,thick] (12) -- (6);

%% file: figures/L2.tex
 ===== NODES ====

 \node[nodestyle] (cup) at (0.0, 3.9) {$\{a,b,c,d\}$};

\node[nodestyle,fill=orange] (ad) at (-2, 2.6) {$\{a,d\}$};
\node[nodestyle,fill=orange] (bd) at (0, 2.6) {$\{b,d\}$};
\node[nodestyle,fill=orange] (cd) at (2, 2.6) {$\{c,d\}$};
\node[nodestyle,fill=orange] (d) at (0, 1.3) {$\{d\}$};

 ===== MOBIUS ====

\node[right of=cup,xshift=.2cm] {$1$};
\node[right of=ad] {$-1$};
\node[right of=bd] {$-1$};
\node[right of=cd] {$-1$};
\node[right of=d,xshift=-.3cm] {$2$};

 ===== EDGES ====

\draw[black,thick] (cup) -- (ad);
\draw[black,thick] (cup) -- (bd);
\draw[black,thick] (cup) -- (cd);
\draw[black,thick] (d) -- (ad);
\draw[black,thick] (d) -- (bd);
\draw[black,thick] (d) -- (cd);

%% file: figures/L3.tex
 ===== NODES ====

 \node[nodestyle] (cup) at (0.0, 3.9) {$\{a,b,c,d\}$};

\node[nodestyle, fill=orange] (d) at (2.85, 2.6) {$\{d\}$};
%-2.85

\node[nodestyle, fill=orange] (ab) at (-2.85, 2.6) {$\{a,b\}$};
\node[nodestyle, fill=orange] (a) at (-2.85, 1.3) {$\{a\}$};
%-0.95

\node[nodestyle, fill=orange] (ac) at (-0.95, 2.6) {$\{a, c\}$};
\node[nodestyle, fill=orange] (b) at (-0.95, 1.3) {$\{b\}$};
% 0.95

\node[nodestyle, fill=orange] (bc) at (0.95, 2.6) {$\{b, c\}$};
\node[nodestyle, fill=orange] (c) at (0.95, 1.3) {$\{c\}$};

\node[nodestyle] (emptyset) at (0.0, 0.0) {$\emptyset$};

 ===== MOBIUS ====

\node[right of=cup,xshift=.2cm] {$1$};
\node[right of=d,xshift=-.2cm] {$-1$};
\node[right of=bc,xshift=-.1cm] {$-1$};
\node[right of=c,xshift=-.3cm] {$1$};

\node[right of=ac,xshift=-0.1cm] {$-1$};
\node[right of=ab,xshift=-0.1cm] {$-1$};
\node[right of=a,xshift=-.3cm] {$1$};
\node[right of=b,xshift=-.3cm] {$1$};

\node[left of=emptyset,xshift=.40cm] {$0$};

 ===== EDGES ====

\draw[black,thick] (cup) -- (bc);
\draw[black,thick] (cup) -- (ac);
\draw[black,thick] (cup) -- (ab);
\draw[black,thick] (cup) -- (d);

\draw[black,thick] (bc) -- (c);
\draw[black,thick] (bc) -- (b);

\draw[black,thick] (ac) -- (c);
\draw[black,thick] (ac) -- (a);

\draw[black,thick] (ab) -- (b);
\draw[black,thick] (ab) -- (a);

\draw[black,thick] (d) to [bend left] (emptyset);
\draw[black,thick] (c) -- (emptyset);
\draw[black,thick] (b) -- (emptyset);
\draw[black,thick] (a) -- (emptyset);

%% file: figures/L4.tex
 ===== NODES ====

 \node[nodestyle] (cup) at (0.0, 3.9) {$\{a,b,c,d,e,f,g,h\}$};

\node[nodestyle,fill=orange] (ah) at (4, 2.6) {$\{a,h\}$};
%-2.85

\node[nodestyle,fill=orange] (acdg) at (-4.5, 2.6) {$\{a,c,d,g\}$};
\node[nodestyle,fill=orange] (ad) at (-4.5, 1.3) {$\{a,d\}$};
%-0.95

\node[nodestyle,fill=orange] (abdf) at (-1.5, 2.6) {$\{a,b,d,f\}$};
\node[nodestyle,fill=orange] (ac) at (-1.5, 1.3) {$\{a,c\}$};
% 0.95

\node[nodestyle,fill=orange] (abce) at (1.5, 2.6) {$\{a,b,c,e\}$};
\node[nodestyle,fill=orange] (ab) at (1.5, 1.3) {$\{a,b\}$};

\node[nodestyle] (a) at (0.0, 0.0) {$\{a\}$};

 ===== MOBIUS ====

\node[right of=cup,xshift=+0.9cm] {$1$};
\node[right of=ah,xshift=0.1cm] {$-1$};
\node[right of=abce,xshift=0.3cm] {$-1$};
\node[right of=ab,xshift=-0.2cm] {$1$};

\node[right of=abdf,xshift=+0.3cm] {$-1$};
\node[right of=acdg,xshift=+0.3cm] {$-1$};
\node[right of=ad,xshift=-0.2cm] {$1$};
\node[right of=ac,xshift=-0.2cm] {$1$};

\node[left of=a,xshift=.3cm] {$0$};

 ===== EDGES ====

\draw[black,thick] (cup) -- (abce);
\draw[black,thick] (cup) -- (abdf);
\draw[black,thick] (cup) -- (acdg);
\draw[black,thick] (cup) -- (ah);

\draw[black,thick] (abce) -- (ab);
\draw[black,thick] (abce) -- (ac);

\draw[black,thick] (abdf) -- (ad);
\draw[black,thick] (abdf) -- (ab);

\draw[black,thick] (acdg) -- (ac);
\draw[black,thick] (acdg) -- (ad);

\draw[black,thick] (ah) to [bend left] (a);
\draw[black,thick] (ad) -- (a);
\draw[black,thick] (ac) -- (a);
\draw[black,thick] (ab) -- (a);

%% file: figures/Lcomplex.tex
 %===== NODES ====

 \node[nodestyle] (cup) at (0, 0) {$abb'cc'dee'ff'gg'$};
\node[right of=cup,xshift=.9cm] {1};

\node[nodestyle,fill=orange] (abb2d) at (0, -1.3) {$abb'd\vphantom{f}$};
\node[right of=abb2d,xshift=-.2cm] {$-1$};
\node[nodestyle,fill=orange] (abc2f2) at (1, -1.3) {$abc'f'$};
\node[right of=abc2f2,xshift=-.2cm] {$-1$};
\node[nodestyle,fill=orange] (ab2c2e2) at (2, -1.3) {$ab'c'e'\vphantom{b'f}$};
\node[right of=ab2c2e2,xshift=-.2cm] {$-1$};
\node[nodestyle,fill=orange] (ag2) at (3, -1.3) {$ag'\vphantom{b'f}$};
\node[right of=ag2,xshift=-.2cm] {$-1$};
\node[nodestyle,fill=orange] (ab2cf) at (-1, -1.3) {$ab'cf\vphantom{b'f}$};
\node[right of=ab2cf,xshift=-.2cm] {$-1$};
\node[nodestyle,fill=orange] (abce) at (-2, -1.3) {$abce\vphantom{b'f}$};
\node[right of=abce,xshift=-.2cm] {$-1$};
\node[nodestyle,fill=orange] (ag) at (-3, -1.3) {$ag\vphantom{g'f}$};
\node[right of=ag,xshift=-.2cm] {$-1$};

\node[nodestyle,fill=orange] (ab2) at (.5, -2.6) {$ab'\vphantom{b'f}$};
\node[right of=ab2,xshift=-.2cm] {2};
\node[nodestyle,fill=orange] (ac2) at (1.5, -2.6) {$ac'\vphantom{b'f}$};
\node[right of=ac2,xshift=-.2cm] {1};
\node[nodestyle,fill=orange] (ab) at (-.5, -2.6) {$ab\vphantom{b'f}$};
\node[right of=ab,xshift=-.2cm] {2};
\node[nodestyle,fill=orange] (ac) at (-1.5, -2.6) {$ac\vphantom{b'f}$};
\node[right of=ac,xshift=-.2cm] {1};

\node[nodestyle] (a) at (0, -3.9) {$a$};
\node[right of=a,xshift=.2cm] {0};

\draw[black,thick] (a) -- (ab);
\draw[black,thick] (a) -- (ac);
\draw[black,thick] (a) -- (ab2);
\draw[black,thick] (a) -- (ac2);
\draw[black,thick] (a) to [bend left] (ag);
\draw[black,thick] (a) to [bend right] (ag2);

\draw[black,thick] (ac) -- (abce);
\draw[black,thick] (ac) -- (ab2cf);
\draw[black,thick] (ac2) -- (ab2c2e2);
\draw[black,thick] (ac2) -- (abc2f2);
\draw[black,thick] (ab) -- (abce);
\draw[black,thick] (ac) -- (abce);
\draw[black,thick] (ab2) -- (ab2cf);
\draw[black,thick] (ac) -- (ab2cf);
\draw[black,thick] (ab2) -- (ab2c2e2);
\draw[black,thick] (ab) -- (abc2f2);
\draw[black,thick] (ab) -- (abb2d);
\draw[black,thick] (ab2) -- (abb2d);
\draw[black,thick] (ag) -- (cup);
\draw[black,thick] (abce) -- (cup);
\draw[black,thick] (ab2cf) -- (cup);
\draw[black,thick] (abb2d) -- (cup);
\draw[black,thick] (abc2f2) -- (cup);
\draw[black,thick] (ab2c2e2) -- (cup);
\draw[black,thick] (ag2) -- (cup);

%  ===== MOBIUS ====
% 
% \node[right of=cup,xshift=+0.9cm] {$1$};
% \node[right of=ah,xshift=0.1cm] {$-1$};
% \node[right of=abce,xshift=0.3cm] {$-1$};
% \node[right of=ab,xshift=-0.2cm] {$1$};
% 
% \node[right of=abdf,xshift=+0.3cm] {$-1$};
% \node[right of=acdg,xshift=+0.3cm] {$-1$};
% \node[right of=ad,xshift=-0.2cm] {$1$};
% \node[right of=ac,xshift=-0.2cm] {$1$};
% 
% \node[left of=a,xshift=.3cm] {$0$};

% ===== EDGES ====
%
%\draw[black,thick] (cup) -- (abce);
%\draw[black,thick] (cup) -- (abdf);
%\draw[black,thick] (cup) -- (acdg);
%\draw[black,thick] (cup) -- (ah);
%
%\draw[black,thick] (abce) -- (ab);
%\draw[black,thick] (abce) -- (ac);
%
%\draw[black,thick] (abdf) -- (ad);
%\draw[black,thick] (abdf) -- (ab);
%
%\draw[black,thick] (acdg) -- (ac);
%\draw[black,thick] (acdg) -- (ad);
%
%\draw[black,thick] (ah) to [bend left] (a);
%\draw[black,thick] (ad) -- (a);
%\draw[black,thick] (ac) -- (a);
%\draw[black,thick] (ab) -- (a);

%% file: figures/Lcomplextree.tex
 \node (root) {$\cupdot$}
   child { node (minus1) {$\minusdot$}
     child { node (cup1) {$\cupdot$}
       child { node (minus2) {$\minusdot$}
         child { node (cup2) {$\cupdot$}
           child { node (minus3) {$\minusdot$}
             child { node (cup3) {$\cupdot$}
               child { node (minus4) {$\minusdot$}
                 child { node (cup4) {$\cupdot$}
                   child { node (minus5) {$\minusdot$}
                     child { node (cup5) {$\cupdot$}
                       child { node (minus6) {$\minusdot$}
                         child {node {$abce$}}
                         child {node {$ab$}}
                     }
                     child {node {$abc'f'$}}
                   }
                   child {node {$ac'$}}
                 }
                 child {node {$ag$}}
               }
               child {node {$ac$}}
             }
             child {node {$ab'cf$}}
           }
           child {node {$ab'$}}
         }
         child {node {$ag'$}}
       }
       child {node {$ab$}}
     }
     child {node {$abb'd$}}
     }
     child {node {$ab'$}}
   }
   child {node {$ab'c'e'$}};

    \node[right of=root,xshift=-3.1cm,text=orange] {$abb'cc'dee'ff'gg'$};
    \node[right of=minus6,xshift=-1.7cm,text=orange] {$ce$};
    \node[right of=cup5,xshift=-2.2cm,text=orange] {$abcec'f'$};
    \node[right of=minus5,xshift=-1.7cm,text=orange] {$bcef'$};
    \node[right of=cup4,xshift=-2.2cm,text=orange] {$abcef'g$};
    \node[right of=minus4,xshift=-1.9cm,text=orange] {$bef'g$};
    \node[right of=cup3,xshift=-2.5cm,text=orange] {$abb'ceff'g$};
    \node[right of=minus3,xshift=-2.5cm,text=orange] {$bceff'g$};
    \node[right of=cup2,xshift=-2.5cm,text=orange] {$abceff'gg'$};
    \node[right of=minus2,xshift=-2.5cm,text=orange] {$ceff'gg'$};
    \node[right of=cup1,xshift=-2.5cm,text=orange] {$abb'cdeff'gg'$};
    \node[right of=minus1,xshift=-2.5cm,text=orange] {$bcdeff'gg'$};

%% file: conj.tex
In this section we state the conjecture and illustrate it with a few toy 
examples. We adapt to our context notation and terminology
from~\cite{hirsch2017disjoint} for disjoint unions and subset complements.

\begin{definition}
  \label{def:dots}
  For two sets~$S,T$, the \emph{disjoint union}~$S\cupdot T$ equals~$S\cup T$
  if~$S\cap T = \emptyset$, else it is undefined.
  We generalize the disjoint union operator to more than two operands in the
  expected way.
  The \emph{subset
  complement}~$S\minusdot T$ equals~$S\setminus T$ if~$T\subseteq S$, else it is
  undefined.
\end{definition}

Notice that

\begin{equation}
  \label{eq:dot-duality}
 S \cupdot T = U  \text{ iff } U\minusdot S = T \text{ iff } U\minusdot T = S.
\end{equation}

\begin{definition}
  \label{def:dotalgebra}
  Let~$\calF$ be a set family.
  The \emph{partial dot-algebra generated by~$\calF$}, denoted~$\bullet(\calF)$,
  is the smallest (for inclusion) set family of subsets of $2^{\bigcup \calF}$ which contains $\emptyset$ as well as all the sets
  of~$\calF$ and is closed under disjoint union and subset complement. Formally,
  it is defined inductively by:
  \begin{enumerate}
    \item First base case: $\emptyset \in \bullet(\calF)$; 
    \item Second base case: for each $X \in \calF$, we have $X \in \bullet(\calF)$;
    \item Induction: for each~$X_1,X_2\in \bullet(\calF)$, if $X_1 \cupdot X_2$
      (resp., $X_1\minusdot X_2$) is well-defined, then we have~$X_1\cupdot X_2 \in
      \bullet(\calF)$ (resp., $X_1\minusdot X_2 \in
      \bullet(\calF)$).
  \end{enumerate}
\end{definition}

Note that when $\calF$ is not empty, then item (1) is redundant as it
follows from (2) and (3): taking any $X\in \calF$, we have $X\in
\bullet(\calF)$ by item (2), and $X \minusdot X = \emptyset \in \bullet(\calF)$
by item~(3). 

Furthermore, observe that~$\bullet(\calF)$ is always finite because~$\calF$ is. 
What we call
a \emph{witnessing tree of~$X\in \bullet(\calF)$} is a rooted ordered tree whose
leaves are annotated with $\emptyset$ or with elements of~$\calF$ and whose internal nodes are
annotated with either~$\cupdot$ or~$\minusdot$ (with the node being binary if it
is is annotated with~$\minusdot$), with the expected semantics. Obviously, we
have~$X \in \bullet(\calF)$ if and only if such a witnessing tree exists.

\begin{example}
  \label{expl:dot-algebra}
  Let $\calF_1 = \{\{a,d\}, \{b,e\}, \{a,b,c\}\}$. Then $\bullet(\calF_1) = \calF_1 \cup
  \{\emptyset, \{a,b,d,e\}\}$. Let $\calF_2 = \{\{a,c\}, \{b,c\}, \{c\}\}$. Then we
  have $\bullet(\calF_2) = 2^{\{a,b,c\}}$. Two trees~$T_0, T'_0$ which witness
  that~$\{a,b,c\}\in \bullet(\calF_2)$ are shown in Figure~\ref{fig:wit-trees}. 
\end{example}

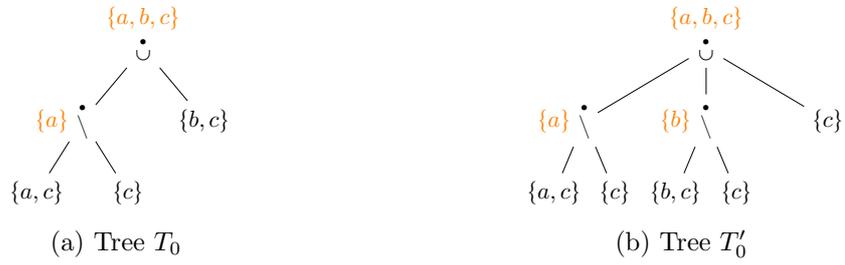
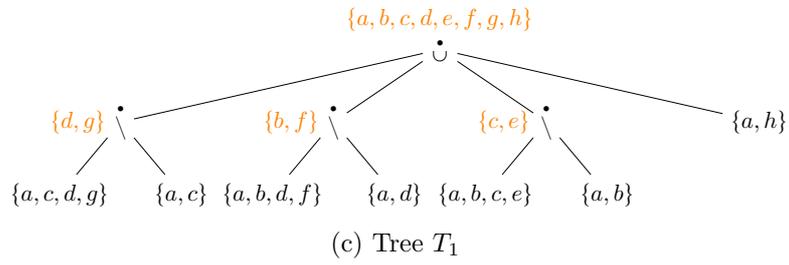
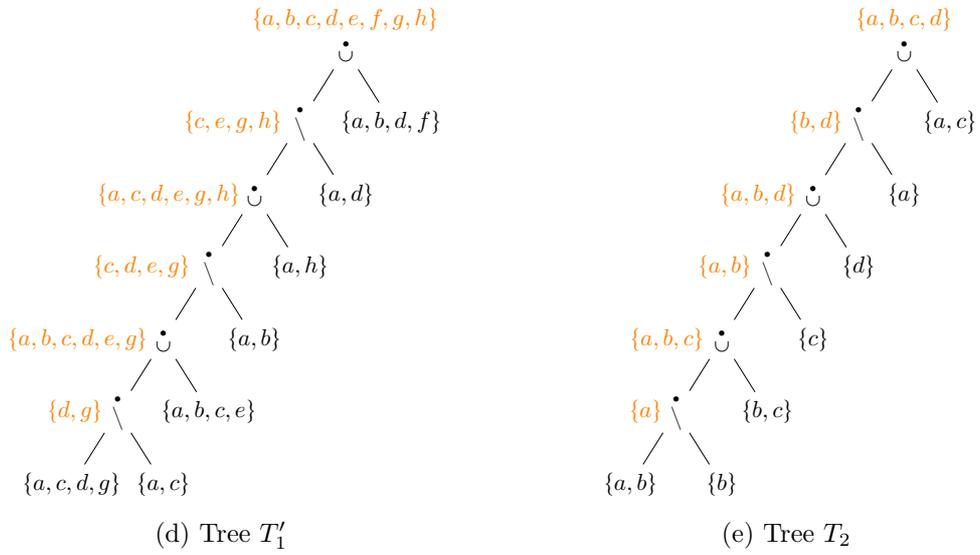
\begin{figure}
\begin{minipage}{\linewidth}

    \begin{minipage}[t]{.5\linewidth}  
    \noindent  
    \begin{subfigure}[b]{\linewidth}
      \centering
      \scalebox{0.8}{
	\begin{tikzpicture}[level distance=1.2cm,
  level 1/.style={sibling distance=2cm},
  level 2/.style={sibling distance=1.5cm}]
	\input{figures/tree0}
    \end{tikzpicture}}
        \caption{Tree~$T_0$}
        \label{fig:T0}
      \end{subfigure}
      \end{minipage}
      \hfill
      \noindent
  \begin{minipage}[t]{.5\linewidth}  
      \noindent  
      \begin{subfigure}[b]{\linewidth}
        \centering
        \scalebox{0.8}{
  	\begin{tikzpicture}[level distance=1.2cm,
    level 1/.style={sibling distance=2cm},
    level 2/.style={sibling distance=1cm}]
          \usetikzlibrary{trees}
  	\input{figures/treep0}
      \end{tikzpicture}}
          \caption{Tree~$T'_0$}
          \label{fig:Tp0}
        \end{subfigure}
        \end{minipage}
\end{minipage}\vspace{1cm}
\begin{minipage}{\linewidth}
      \begin{subfigure}[b]{\linewidth}
        \centering
        \scalebox{0.8}{
  	\begin{tikzpicture}[level distance=1.2cm,
    level 1/.style={sibling distance=3.5cm},
    level 2/.style={sibling distance=2cm}]
          \usetikzlibrary{trees}
  	\input{figures/tree1}
      \end{tikzpicture}}
          \caption{Tree~$T_1$}
          \label{fig:T1}
        \end{subfigure}
\end{minipage}\vspace{1cm}
\begin{minipage}{\linewidth}
  \begin{minipage}{.5\linewidth}
      \begin{subfigure}[b]{\linewidth}
        \centering
        \scalebox{0.8}{
  	\begin{tikzpicture}[level distance=1.2cm,
    level all/.style={sibling distance=3cm}]
          \usetikzlibrary{trees}
  	\input{figures/treep1}
      \end{tikzpicture}}
          \caption{Tree~$T'_1$}
          \label{fig:Tp1}
        \end{subfigure}
  \end{minipage}
  \hfill\noindent
  \begin{minipage}{.5\linewidth}
      \begin{subfigure}[b]{\linewidth}
        \centering
        \scalebox{0.8}{
  	\begin{tikzpicture}[level distance=1.2cm,
    level all/.style={sibling distance=3cm}]
          \usetikzlibrary{trees}
  	\input{figures/tree2}
      \end{tikzpicture}}
          \caption{Tree~$T_2$}
          \label{fig:T2}
        \end{subfigure}
  \end{minipage}
\end{minipage}
  \caption{Various witnessing trees. The sets that internal nodes
  correspond to are shown in orange.}
  \label{fig:wit-trees}
\end{figure}

We are now ready to state
the conjecture.

\begin{conjecture}[$\nci$ conjecture, formulation I]
  \label{conj:general}
  For every intersection lattice~$L$, we have~$\onehat \in
  \bullet(\nci(L))$.
\end{conjecture}

In other words, the conjecture states that, given any non-trivial set family, we
can always express the union of the sets by starting with the set intersections
that do not cancel in the inclusion-exclusion formula and applying the disjoint
union and subset complement operators.

\begin{example}
  \label{expl:conj-expls}
  
  Continuing Example~\ref{expl:intersection-lattices}, the nodes corresponding to sets of~$\nci(L_i)$
  are colored in orange in Figure~\ref{fig:int-latt}: they are the nodes
  different from $\onehat$ with a non-zero value of the Möbius function. We show next that~$\onehat \in \bullet(\nci(L_i))$ for all
  of them.

  \begin{description}
    \item[\textbf{$L_1$.}]
      We have~$\onehat = \{n\in \NN \mid n \equiv 0 \pmod 2 \text{ or } n
      \equiv 0 \pmod 3\} = [\{n \in \NN \mid n \equiv 0 \pmod 2\} \minusdot \{n
      \in \NN \mid n \equiv 0 \pmod 6\}] \cupdot \{n \in \NN \mid n \equiv 0
      \pmod 3\}$.
    \item[\textbf{$L_2$.}]
      We have~$\onehat = \{a,b,c,d\} = [\{a,d\} \minusdot \{d\}] \cupdot
      [\{b,d\} \minusdot \{d\}] \cupdot [\{c,d\}]$.
    \item[\textbf{$L_3$.}] 
      We have~$\onehat = \{a,b,c,d\} = \{a\} \cupdot \{b\} \cupdot \{c\}
      \cupdot \{d\} $.
    \item[\textbf{$L_4$.}]  Two trees~$T_1,T'_1$
      which witness that~$\onehat \in \bullet(\nci(L_4))$ can be found in
      Figure~\ref{fig:wit-trees}.
    \item[\textbf{$L_5$.}]  A witnessing tree for~$L_5$ can be found in
      Figure~\ref{fig:complex}.
  \end{description}
\end{example}

The motivation for the conjecture is to understand whether the
inclusion-exclusion formula can be understood using the Boolean operations of
disjoint union and subset complement. Specifically, let $\calF$ be a
non-trivial set family, consider its intersection lattice~$\bbmL_\calF$, recall that
$\onehat = \bigcup \calF$, and let
$\xi$ be a measure on~$2^{\onehat}$. We know that the measure 
of~$\onehat$ can be obtained from the measure of the non-cancelling
intersections, namely, Proposition~\ref{prp:incl-excl-mob} expresses the measure
$\xi(\onehat)$ of the set~$\onehat$ as an arithmetic combination of the measures $\xi(S)$ with $S \in \nci(L)$.
The conjecture says that we can also express the \emph{set} $\onehat$ from
the \emph{sets} $S$ with $S \in \nci(L)$ using the Boolean operations of disjoint union
and subset complement. 

\paragraph*{Dual version: non-cancelling unions conjecture.}
We end this section by presenting a dual version of the $\nci$
conjecture that focuses on unions instead of intersections, and that we
call the \emph{non-cancelling unions} ($\ncu$) conjecture.
Call a set family~$\calF$ \emph{co-trivial} if it is empty or if there is~$X\in \calF$ such
that~$X = \bigcap \calF$, and \emph{non-co-trivial otherwise}. Define the
\emph{union lattice}~$\bbmU_\calF$ of a non-co-trivial set family~$\calF$ as follows:
let~$R_\emptyset \defeq \bigcap \calF$ and~$R_\calT \defeq \bigcup
\calT$ for~$\calT \subseteq \calF$, $\calT\neq \emptyset$, and~$\bbmU_\calF \defeq (\{R_\calT \mid \calT \subseteq \calF\},
\subseteq)$. Then~$\bbmU_\calF$ is a lattice whose join operation corresponds to set union, has a greatest element~$\onehat = \bigcup \calF$ and a least element~$\zerohat = \bigcap \calF$ that are distinct. 
A lattice is a union lattice if it is of the form~$\bbmU_\calF$ for some non-co-trivial set family.
Analogously to Proposition~\ref{prp:incl-excl-mob}, it is easy to show that
for a union lattice~$L$ and measure~$\xi$ on~$2^{\onehat}$, we have
  \begin{equation}\label{eq:incl-excl-mob-unions}
  \xi(\zerohat) = - \sum_{\substack{U\in L \\ U \neq \zerohat}}  \mu_L(\zerohat,U) \xi(U).
\end{equation}
Define then the
\emph{non-cancelling (non-trivial) unions of~$L$} as 
\[\ncu(L) \defeq \{U \in L \mid U\neq \zerohat, \mu_L(\zerohat,U) \neq 0  \}.\]
The non-cancelling unions conjecture is then:

\begin{conjecture}[$\ncu$ conjecture]
  \label{conj:ncu}
  For every union lattice~$L$, we have~$\zerohat \in
  \bullet(\ncu(L))$.
\end{conjecture}
Notice here that the definition of the dot-algebra is unchanged: we
still use disjoint complement and disjoint union. 
Informally again, this is stating that, given a set family, we can always express their intersection by applying the disjoint union and subset complement operations to the set unions that 
do not cancel in the (“union form” of the) inclusion-exclusion formula.

We show in Appendix~\ref{apx:conj} that, unsurprisingly, the~$\nci$ and~$\ncu$ conjectures 
are in fact equivalent. We
decided to focus on the intersections version of the conjecture
because inclusion-exclusion is most often presented in 
“intersection form”.

\begin{toappendix}
  We show in this section that the~$\nci$ and~$\ncu$ conjectures
  are equivalent. We only sketch the proofs as details are easy to
  fill.
  \begin{definition}
    \label{def:dual}
    Let~$\calF$ be a set family. For~$S\in 2^{\bigcup \calF}$,
    let~$\dual_\calF(S) \defeq (\bigcup \calF) \setminus S$, and
    for~$\calT\subseteq \calF$
    let~$\dual_\calF(\calT) \defeq \{\dual_\calF(X) \mid X \in \calT\}$.
  \end{definition}

  Observe then that~$\calF$ is non-trivial if and only
  if~$\dual_\calF(\calF)$ is non-co-trivial.

  \begin{fact}
    \label{fact:compl}
    Let~$L=\bbmL_\calF$ be an intersection lattice, and let~$L''$
    be the union lattice~$\bbmU_{\dual_\calF(\calF)}$
    of~$\dual_\calF(\calF)$, but with the order reversed. Then we
    have~$L\simeq L''$.
  \end{fact}
  \begin{proofsketch}
    One can check that the function that sends~$\onehat_L$
    to~$\onehat_{L''}$ ($=\zerohat_{\bbmU_{\dual_\calF(\calF)}}$) and~$S_\calT =\bigcap \calT$ for~$\calT
    \subseteq \calF$, $\calT\neq \emptyset$
    to~$R_{\dual_\calF(\calT)} = \bigcup\dual_\calF(\calT)$ is an
    isomorphism between~$L$ and~$L''$. This uses the following two
    simple facts:
    \begin{enumerate}
      \item For any sets~$A,B,C$ with~$B\cup C \subseteq A$ we have
        $A\setminus B \subseteq A \setminus C$ if and only if~$C
        \subseteq B$.
      \item For a set~$A$ and set family~$\calT$ we
        have~$\bigcap_{X\in \calT} A\setminus X = A \setminus
        \bigcup \calT$.\qedhere
      \end{enumerate}
  \end{proofsketch}
  The dual version of this fact also holds, i.e., for a union
  lattice $L=\bbmU_\calF$, letting~$L''$ be the intersection
  lattice $\bbmL_{\dual_\calF(\calF)}$ of~$\dual_\calF(\calF)$ but
  with the order reversed, we have~$L\simeq L''$. (This uses point (1) above and
  the dual of point (2), namely,
  $\bigcup_{X\in \calT} A\setminus X = A \setminus \bigcap
  \calT$.)

  We can then show that the~$\ncu$ conjecture implies the~$\nci$
  conjecture as follows. Let~$L = \bbmL_\calF$ be an intersection
  lattice, and let~$L' \defeq \bbmU_{\dual_\calF(\calF)}$ be the
  union lattice of~$\dual_\calF(\calF)$. Using
  Fact~\ref{fact:compl} and \cite[Proposition
  3]{rota1964foundations}, it is routine to show by induction
  on~$\bullet(\ncu(L'))$ the following claim ($\dagger$): for every~$R\in
  \bullet(\ncu(L'))$ we have~$\dual_\calF(R) \in \bullet(\nci(L))$.
  Since the~$\ncu$ conjecture holds we have~$\zerohat_{L'} \in
  \bullet(\ncu(L'))$, hence by~($\dagger$) we have
  $\dual_\calF(\zerohat_{L'}) \in \bullet(\nci(L))$. But observe
  that~$\zerohat_{L'} = \bigcap_{X\in \calF} \dual_\calF(X)$ by
  definition, which is equal to~$\bigcup \calF \setminus \bigcup
  \calF = \emptyset$ by Item (2) above.
  So~$\dual_\calF(\zerohat_{L'}) = \bigcup \calF$, and this
  is~$\onehat_L$, so indeed~$\onehat_L \in \nci(L)$.

  We can show that the~$\nci$ conjecture implies the~$\ncu$
  conjecture in the same way, using the dual version of
  Fact~\ref{fact:compl}. This completes the proof.
\end{toappendix}

%% file: figures/tree0.tex
  \node (root) {$\cupdot$}
    child {node (firstminus) {$\minusdot$}
      child {node {$\{a,c\}$}}
      child {node {$\{c\}$}}
    }
    child {node {$\{b,c\}$}};

    \node[above of=root,yshift=-.5cm,text=orange] {$\{a,b,c\}$};
    \node[right of=firstminus,xshift=-1.5cm,text=orange] {$\{a\}$};

%% file: figures/treep0.tex
  \node (root) {$\cupdot$}
    child {node (firstminus) {$\minusdot$}
      child {node {$\{a,c\}$}}
      child {node {$\{c\}$}}
    }
    child {node (secondminus) {$\minusdot$}
      child {node {$\{b,c\}$}}
      child {node {$\{c\}$}}
    }
    child {node {$\{c\}$}};
    \node[above of=root,yshift=-.5cm,text=orange] {$\{a,b,c\}$};
    \node[right of=firstminus,xshift=-1.5cm,text=orange] {$\{a\}$};
    \node[right of=secondminus,xshift=-1.5cm,text=orange] {$\{b\}$};

%% file: figures/tree1.tex
  \node (root) {$\cupdot$}
    child {node (minus1) {$\minusdot$}
      child {node {$\{a,c,d,g\}$}}
      child {node {$\{a,c\}$}}
    }
    child {node (minus2) {$\minusdot$}
      child {node {$\{a,b,d,f\}$}}
      child {node {$\{a,d\}$}}
    }
    child {node (minus3) {$\minusdot$}
      child {node {$\{a,b,c,e\}$}}
      child {node {$\{a,b\}$}}
    }
    child {node {$\{a,h\}$}};
    \node[above of=root,yshift=-.50cm,text=orange] {$\{a,b,c,d,e,f,g,h\}$};
    \node[right of=minus1,xshift=-1.7cm,text=orange] {$\{d,g\}$};
    \node[right of=minus2,xshift=-1.7cm,text=orange] {$\{b,f\}$};
    \node[right of=minus3,xshift=-1.7cm,text=orange] {$\{c,e\}$};

%% file: figures/treep1.tex
  \node (root) {$\cupdot$}
    child {node (minus1) {$\minusdot$}
      child {node (cup1) {$\cupdot$}
        child { node (minus2) {$\minusdot$}
          child { node (cup2) {$\cupdot$}
            child { node (minus3) {$\minusdot$}
              child {node {$\{a,c,d,g\}$}}
              child {node {$\{a,c\}$}}
            }
            child {node {$\{a,b,c,e\}$}}
          }
          child {node {$\{a,b\}$}}
        }
        child {node {$\{a,h\}$}}
        }
      child {node {$\{a,d\}$}}
    }
    child {node {$\{a,b,d,f\}$}};
    \node[above of=root,yshift=-.50cm,text=orange] {$\{a,b,c,d,e,f,g,h\}$};
    \node[right of=minus3,xshift=-1.7cm,text=orange] {$\{d,g\}$};
    \node[right of=cup2,xshift=-2.4cm,text=orange] {$\{a,b,c,d,e,g\}$};
    \node[right of=minus2,xshift=-2.1cm,text=orange] {$\{c,d,e,g\}$};
    \node[right of=cup1,xshift=-2.4cm,text=orange] {$\{a,c,d,e,g,h\}$};
    \node[right of=minus1,xshift=-2.1cm,text=orange] {$\{c,e,g,h\}$};

%% file: figures/tree2.tex
  \node (root) {$\cupdot$}
    child {node (minus1) {$\minusdot$}
      child {node (cup1) {$\cupdot$}
        child { node (minus2) {$\minusdot$}
          child { node (cup2) {$\cupdot$}
            child { node (minus3) {$\minusdot$}
              child {node {$\{a,b\}$}}
              child {node {$\{b\}$}}
            }
            child {node {$\{b,c\}$}}
          }
          child {node {$\{c\}$}}
        }
        child {node {$\{d\}$}}
        }
      child {node {$\{a\}$}}
    }
    child {node {$\{a,c\}$}};

    \node[above of=root,yshift=-.50cm,text=orange] {$\{a,b,c,d\}$};
    \node[right of=minus3,xshift=-1.5cm,text=orange] {$\{a\}$};
    \node[right of=cup2,xshift=-1.9cm,text=orange] {$\{a,b,c\}$};
    \node[right of=minus2,xshift=-1.7cm,text=orange] {$\{a,b\}$};
    \node[right of=cup1,xshift=-1.9cm,text=orange] {$\{a,b,d\}$};
    \node[right of=minus1,xshift=-1.7cm,text=orange] {$\{b,d\}$};

%% file: tight.tex
In this section we show that the $\nci$ conjecture 
can be studied without loss of generality on intersection lattices
satisfying a certain \emph{tightness} condition, that we define
next. Intuitively, an intersection lattice~$L$ is \emph{tight} if for
every~$U\in L$, $U\neq \onehat$, there is exactly one element~$x\in
\onehat$ such that~$\min_L(x) = U$. 
We also define the notion of an intersection lattice being
\emph{full}, which will be useful in the next section.

\begin{definition}
  \label{def:tight}
  An intersection lattice~$L$ is \emph{full} if for every~$U\in L$,~$U\neq
  \onehat$, we have~$\{x\in \onehat \mid \min_L(x) = U\} \neq \emptyset$.\footnote{Note that~$\{x\in
\onehat \mid \min_L(x) = \onehat\} = \emptyset$ because $\calF$ is
non-trivial.}
  The lattice~$L$ is in addition called \emph{tight} if for every~$U\in L$,~$U\neq
  \onehat$, we have~$|\{x\in \onehat \mid \min_L(x) = U\}| = 1$. 
\end{definition}

For instance, the intersection lattices~$L_2$ and~$L_4$ from
Figure~\ref{fig:int-latt} and~$L_5$ from Figure~\ref{fig:complex} are tight (hence full),
$L_1$ is full but not tight, and $L_3$ is not full (hence not
tight). Note that if an intersection lattice~$L=\bbmL_\calF$ is
tight, then~$\onehat$ is finite, i.e., the sets in~$\calF$ are
themselves finite: this can be seen by applying
Fact~\ref{fact:union-min} with~$U = \onehat$. 
We then claim that we can reformulate the $\nci$ conjecture as
follows:

\begin{conjecture}[$\nci$ conjecture, formulation II]
  \label{conj:tight}
  For every \emph{tight} intersection lattice~$L$, we have~$\onehat \in
  \bullet(\nci(L))$.
\end{conjecture}

We prove this in two steps.
First, we show that for every intersection lattice~$L$, there
exists a \emph{tight} intersection lattice~$L'$ that is isomorphic
to~$L$. 

\begin{lemma}
  \label{lem:to-tight}
  For every intersection lattice~$L$, there exists a tight
  intersection lattice~$L'$ such that~$L\simeq L'$.
\end{lemma}
\begin{proof}
  It suffices to observe that the intersection lattice constructed in
  Remark~\ref{rmk:lattices-are-intersection} (call it $L'$) is tight. 
  Indeed, letting $U\in L\setminus \{\onehat_L\}$, corresponding to $\downarrow_L(\{U\})$ in $L'$,
  we have that $\{x\in \onehat_{L'} \mid \min_{L'}(x) = \downarrow_L(\{U\})\} = \{U\}$.
\end{proof}

Second, we show that, for two isomorphic intersection
lattices~$L,L'$ such that~$L'$ is full, then~$\onehat_{L'} \in
\bullet(\nci(L'))$ implies $\onehat_L \in \bullet(\nci(L))$. 

\begin{lemma}
  \label{lem:from-full}
  Let $L,L'$ be isomorphic intersection lattices such that~$L'$ is full and such that
  $\onehat_{L'} \in \bullet(\nci(L'))$. Then we have $\onehat_L \in
  \bullet(\nci(L))$.
\end{lemma}
\begin{proof}
  Let~$h:L\to L'$ be an isomorphism. 
  Since~$L'$ is full, for~$U'\in L'$, $U' \neq \onehat_{L'}$, let
  $\alpha_{U'}$ be an element of $\{x'\in
  \onehat_{L'} \mid \min_{L'}(x') = U'\}$.
  Define~$g:\onehat_L \to \onehat_{L'}$
  by~$g(x) \defeq \alpha_{h(\min_{L}(x))}$ for~$x\in \onehat_L$. Note that~$g$
  is not necessarily injective. For~$x' \in \onehat_{L'}$ let~$g^{-1}(x')
  \defeq \{x\in \onehat_L \mid g(x) = x'\}$, and for~$X'\subseteq \onehat_{L'}$
  let $g^{-1}(X') \defeq \bigcup_{x'\in X'} g^{-1}(x')$. From
  Fact~\ref{fact:union-min} and because~$L'$ is full it is easy to see that
  ($\dagger$) for every~$U'\in L'$ we have~$g^{-1}(U') = h^{-1}(U')$. In
  particular we have~$g^{-1}(\onehat_{L'}) = \onehat_L$. Hence, it is enough to
  prove that for all~$X' \in \bullet(\nci(L'))$ we have~$g^{-1}(X')\in
  \bullet(\nci(L))$. We show this by induction on~$\nci(L')$. 
  The first base case is when~$X' = \emptyset$. But then observe
  that~$g^{-1}(X') = \emptyset$, and that~$\emptyset \in \bullet(\nci(L))$ by definition.
  The second base case is
  when~$X' \in \nci(L')$. But then~$g^{-1}(X') \in \nci(L)$ as well by
  ($\dagger$) and because~$h$ is an isomorphism so
  clearly~$\mu_{L'}(X',\onehat_{L'}) = \mu_L(h^{-1}(X'),\onehat_L)$. The first
  inductive case is when~$X' = X'_1 \cupdot X'_2$ with~$X'_1,X'_2 \in
  \bullet(\nci(L'))$. By induction hypothesis we have~$g^{-1}(X'_1),
  g^{-1}(X'_2) \in \bullet(\nci(L))$. Moreover clearly~$g^{-1}(X'_1)\cap
  g^{-1}(X'_2) = \emptyset$ since~$X'_1$ and~$X'_2$ are disjoint. So
  indeed~$g^{-1}(X') = g^{-1}(X'_1) \cupdot g^{-1}(X'_2) \in \bullet(\nci(L))$.
  The second inductive case is when~$X' = X'_1 \minusdot X'_2$ with~$X'_1,X'_2
  \in \bullet(\nci(L'))$. By induction hypothesis again we have~$g^{-1}(X'_1),
  g^{-1}(X'_2) \in \bullet(\nci(L))$. Moreover clearly~$g^{-1}(X'_2)\subseteq
  g^{-1}(X'_1)$ since~$X'_2 \subseteq X'_1$. So indeed~$g^{-1}(X') =
  g^{-1}(X'_1) \minusdot g^{-1}(X'_2) \in \bullet(\nci(L))$. This concludes the
  proof.
\end{proof}

This indeed proves that the two formulations are equivalent. We illustrate this last
construction in the following example.

\begin{example}
  Consider the intersection lattices~$L_3$ and~$L_4$ from
  Figure~\ref{fig:int-latt}, the isomorphism $h:L_3 \to L_4$ defined
  following their drawings, and~$\alpha_{U'}$ for~$U' \in L_4$~$U' \neq \onehat_{L_4}$ being the only element in $\{x' \in \onehat_{L_4} \mid \min_{L_4}(x') = U'\}$ (since~$L_4$ is tight). 
  Then we have,
  for instance, $g^{-1}(d) = \{a\}$,~$g^{-1}(a) = g^{-1}(f) = \emptyset$.
  Moreover, consider the tree~$T'_1$ from Figure~\ref{fig:Tp1} witnessing
  that~$\onehat_{L_4} \in \bullet(\nci(L_4))$. Applying~$g^{-1}$ to every node
  of this tree yields the tree~$T_2$ from Figure~\ref{fig:T2} that witnesses
  $\onehat_{L_4} \in \bullet(\nci(L_4))$. By contrast, consider the
  expression $\{a\} \cupdot \{b\} \cupdot \{c\} \cupdot \{d\}$ that also
  witnesses $\onehat_{L_3} \in \bullet(\nci(L_3))$. This expression cannot be
  translated through~$h$ into a tree for~$L_4$ as, e.g., the
  operation~$\{b\} \cupdot \{c\}$ would translate to~$\{a,c\} \cupdot \{a,b\}$
  in~$L_4$, which is not valid.
\end{example}

In fact we claim that, when~$L$ is a full intersection lattice, any
tree~$T$ witnessing that~$\onehat \in \bullet(\nci(L))$ must use all
of the sets in~$\nci(L)$. (This is not the case of the expression
$\{a\} \cupdot \{b\} \cupdot \{c\} \cupdot \{d\}$ for~$L_3$ for
instance, but~$L_3$ is not full.)  
  In a sense, this means that the full
intersection lattices are the hardest cases for the $\nci$ conjecture.
We formalize and state a stronger claim (on the non-trivial
intersections instead of only the non-cancelling ones) in the
remaining of this section.

\begin{definition}
  \label{multiplicity}
  Let~$\calF$ be a set family,~$X\in \bullet(\calF)$ and~$T$ a tree
  witnessing this. For a leaf~$\ell$ of~$T$, the \emph{polarity of~$\ell$
  in~$T$}, denoted~$\pol_T(\ell)$, is $1$ if the number of times we take the
  right edge of a~$\minusdot$ node in the path from the root of~$T$ of~$\ell$
  is even; and~$-1$ otherwise. The \emph{multiplicity of~$U\in \calF \cup \{\emptyset\}$ in~$T$}
  is then
\[\mult_T(U) \defeq \sum_{\substack{\ell \text{ leaf of T}\\\text{of the form } U}} \pol_T(\ell).\]
\end{definition}

In other words, if we see each~$U\in \calF \cup \{\emptyset\}$ as a variable and
see~$T$ as an arithmetic expression over these variables (by
replacing~$\cupdot$ nodes by~$+$ and~$\minusdot$ nodes by~$-$), then
the multiplicity of~$U$ in~$T$ is the coefficient of the
corresponding variable once we develop this expression and factorize.
We then have:

\begin{propositionrep}
  \label{prp:full-multiplicity}
   Let~$L$ be a full intersection lattice such that~$\onehat \in
   \bullet(\nti(L))$, and~$T$ a tree witnessing this. Then we have~$\mult_T(U)
   = - \mu_L(U,\onehat)$ for every~$U\in \nti(L)$.
\end{propositionrep}
\begin{proof}
  For~$U\in L$, $U\neq \onehat$, let~$\alpha_U$ be an
  element of $\{x\in \onehat \mid \min_L(x) = U\}$.
  We first define, for each $X\subseteq \onehat$, a function
  $\mucheck_{L,X}:L \to \ZZ$, and state three useful properties of
  these functions. The function $\mucheck_{L,X}$ for~$X\subseteq
  \onehat$ is  defined by $\mucheck_{L,X}(\onehat) \defeq 0$, and,
  for~$U\in \nti(L)$, by~$\mucheck_{L,X}(U) \defeq
  1-\sum_{\substack{U' \in L\\ U \subsetneq U'}}
  \mucheck_{L,X}(U')$ if~$\alpha_U \in X$, and $\mucheck_{L,X}(U)
  \defeq -\sum_{\substack{U' \in L\\ U \subsetneq U'}}
  \mucheck_{L,X}(U')$ if~$\alpha_U \notin X$. 
  
  The three properties are: 
  \begin{enumerate}
    \item \label{itemzero} We have $\mucheck_{L,\emptyset}(U) = 0$ for~$U\in L$.
    \item \label{itema} For every~$U\in \nti(L)$ we
      have~$\mucheck_{L,\onehat}(U) = - \mu_L(U,\onehat)$. 
    \item \label{itemb} For~$X_1, X_2 \subseteq \onehat$ such that~$X_1 \cap X_2 =
     \emptyset$ (resp., such that~$X_2 \subseteq X_1$) we have
     $\mucheck_{L,X_1\cupdot X_2} = \mucheck_{L,X_1} +
     \mucheck_{L,X_2}$ (resp., $\mucheck_{L,X_1\minusdot X_2} =
     \mucheck_{L,X_1} - \mucheck_{L,X_2}$). 
   \item \label{itemc} For~$U\in \nti(L)$ and~$U'\in
     L$, the value $\mucheck_{L,U}(U')$ is $1$ if~$U'=U$ and~$0$
     otherwise.
  \end{enumerate}
  These can all easily be established by top-down induction on~$L$.
  Now for a node~$n \in T$, let~$X_n$ be the subset of~$\onehat$
  represented by the subtree of~$T$ rooted at~$n$. We show by
  bottom-up induction on~$T$ that for every node~$n$ of~$T$
  and~$U\in \nti(L)$ we have~$\mult_{T_n}(U) =
  \mucheck_{L,X_n}(U)$. This will indeed imply the result by applying
  it to the root of~$T$ and using item (\ref{itema}), since we then
  have that~$X_n = \onehat$. The base case, when~$n$ is a leaf
  of~$T$, follows from item (\ref{itemc}) and item (\ref{itemzero}), while the inductive
  case, when~$n$ is an internal node of~$T$, follows
  from~(\ref{itemb}). 
\end{proof}

Since Proposition~\ref{prp:full-multiplicity} will not be strictly
needed to establish our main results, we defer its proof to
Appendix~\ref{apx:tight}. (We include this proposition as we
think that it is relevant to motivate the conjecture.)
Notice that Proposition~\ref{prp:full-multiplicity} can be false if we
do not impose intersection lattices to be full, see e.g., the expression
for~$L_3$ given in Example~\ref{expl:conj-expls}.

Furthermore, one can show that for any tight intersection lattice~$L$
we have~$\bullet(\nti(L)) = 2^{\onehat_L}$. In other words, intuitively, if we
use the non-trivial intersections instead of the non-canceling intersections
in the conjecture, then we can achieve all possible sets of elements:

\begin{fact}
  \label{fact:annoyingfact}
For every tight intersection lattice~$L$
we have~$\bullet(\nti(L)) = 2^{\onehat_L}$.
\end{fact}

\begin{proof}
  Generalizing from tight intersection lattices,
  we will show that the result holds more generally when we assume that for
  every $U \in L$, $U \neq \onehat$, there is \emph{at most} one element $x \in
  \onehat$ such that $\min_L(x) = U$.
  Let $L$ be an intersection lattice satisfying this weaker condition. We
  proceed by induction on the downsets of $L$, showing the following by
  induction on $0 \leq n < |L|$: for any downset $D$ of $L$ containing $n$ nodes, letting
  $X_D \subseteq \onehat$ be the set of elements introduced at nodes of~$D$,
  formally, $\{x \in \onehat \mid \min_L(x) \in D\}$, then we have $\bullet(D) =
  2^{X_D}$. Note that, when taking this claim with $n = |L|-1$, then the only
  downset $D$ of~$L$ containing $|L|-1$ nodes is the downset $D$ containing all
  intersections except $S_\emptyset$, i.e., precisely $\nti(L)$, and then $X_D =
  \onehat$, so this claim with $n = |L|-1$ is the result we want to show.

  The base case is $n = 0$, i.e., the downset $D=\emptyset$. Then~$X_D =
  \emptyset$ and $\bullet(\emptyset)=\{\emptyset\} = 2^{X_D}$ and we are done.

  For the induction case, let $D$ be a downset containing at least one node,
  let $U$ be some maximal node of~$D$, and let $D' = D \setminus \{U\}$. The
  first case is then
  there is no element introduced at~$U$, i.e., $U = \bigcup_{U' \in L, U'
  \subsetneq U} U'$, then we have $X_D = X_{D'}$, and the induction hypothesis
  gives us that $\bullet(D') = 2^{X_{D'}} = 2^{X_D}$, which concludes because
  $\bullet(D) \supseteq \bullet(D')$ (as $D' \subseteq D$)
  and clearly $\bullet(D) \subseteq 2^{X_D}$.
  Let us thus focus on the second case, where there is an element $x$ introduced
  at~$U$.

  We first show that $\{x\} \in \bullet(D)$. Indeed, by definition we have
  $U \in \bullet(D)$ because $U \in D$. 
  Further, all elements of~$U$
  except~$x$ were introduced below~$U$ and in~$D'$, and therefore 
  we have that $U \setminus \{x\} \subseteq X_{D'}$. Hence, by induction
  hypothesis, we have $U \setminus \{x\} \in \bullet(D')$, so $U \setminus \{x\}
  \in \bullet(D)$ because $D' \subseteq D$. From $U \in
  \bullet(D)$ and $U \setminus \{x\} \in \bullet(D)$ we deduce $\{x\} \in
  \bullet(D)$.

  Now let us show that $\bullet(D) = 2^{X_D}$. Pick $Y \in 2^{X_D}$. If $x
  \notin Y$, then $Y \in 2^{X_{D'}}$, so by induction hypothesis $Y \in
  \bullet(D')$, hence $Y \in \bullet(D)$, so we conclude immediately. If $x \in
  Y$, let $Y' = Y \setminus \{x\}$. We have $Y' \in 2^{X_{D'}}$, so by induction
  hypothesis we have $Y' \in \bullet(D')$ hence $Y' \in \bullet(D)$. Further, we
  have established in the previous paragraph that $\{x\} \in \bullet(D)$. We
  conclude that $Y \in \bullet(D)$. Thus, we have shown the inductive claim
  $\bullet(D) = 2^{X_D}$.
  
  We have concluded the induction, which as we explained establishes the claim.
\end{proof}

By Proposition~\ref{prp:full-multiplicity}, in a tree~$T$
which witnesses that~$\onehat \in \bullet(\nti(L))$ for a full intersection
lattice~$L$, the multiplicity of a cancelling intersection is
zero. Intuitively, we can see this as suggesting that we do not
really need such intersections, and so that the $\nci$ conjecture
should morally be true. 
In addition notice that, for a measure~$\xi$ on~$2^{\onehat}$,
applying~$\xi$ to~$T$ (and replacing~$\cupdot$ nodes by~$+$
and~$\minusdot$ nodes by~$-$) and then factoring gives back exactly
Equation~\eqref{eq:incl-excl-mob} according to
Proposition~\ref{prp:full-multiplicity}.

%% file: alt.tex
In this section we give an alternative formulation of the
conjecture which is about the Boolean lattice. It will in
particular use results from the previous section. As the structure of the
Boolean lattice is ``simpler'', the hope is that 
the alternative formulation of the conjecture may be
easier to attack. 
We present the alternative formulation and show that it is
equivalent in Section~\ref{subsec:alt}, and show in
Section~\ref{subsec:useful} some useful properties that we can deduce
from this alternative formulation.

\subsection{The $\ncpd$ conjecture}
\label{subsec:alt}

For~$S$ finite, we call a \emph{configuration of~$\bbmB_S$}, or
simply \emph{configuration} when clear from context, a 
subset of~$\bbmB_S$. 
 Since configurations are in particular set
 families, we will generally denote them with cursive letters.

We start by defining what we call \emph{generalized Möbius
functions}, that are parameterized by a configuration, and that
take only one argument.

\begin{definition}
\label{def:muhat}
Let~$S$ finite and~$\calC \subseteq \bbmB_S$ a configuration.
We define the \emph{(generalized) Möbius function~$\muhat_{\bbmB_S,\calC} : \bbmB_S \to \ZZ$ associated
to~$\calC$} by top-down induction by:
\[\muhat_{\bbmB_S,\calC}(X)\defeq \begin{cases}
  1 - \sum_{X \subsetneq X'} \muhat_{\bbmB_S,\calC}(X')\text{ if }X \in \calC\\
                           - \sum_{X \subsetneq X'} \muhat_{\bbmB_S,\calC}(X')\text{ if }X \notin \calC
                          \end{cases}. \]
In particular, the value $\muhat_{\bbmB_S,\calC}(S)$ is~$0$ if $S \notin
\calC$, and $1$ if $S \in \calC$.
\end{definition}

\begin{example}
\label{expl:0111001001110010}
    Let~$\calC$ be the configuration \[\{\{0\}, \{1\}, \{0,1\}, \{0,3\},
    \{1,2\}, \{1,3\}, \{0,1,3\}, \{1,2,3\}  \}\] of~$\bbmB_{\{0,1,2,3\}}$. We
    have depicted in Figure~\ref{fig:0111001001110010}
    this configuration and
    its associated Möbius function.
\end{example}

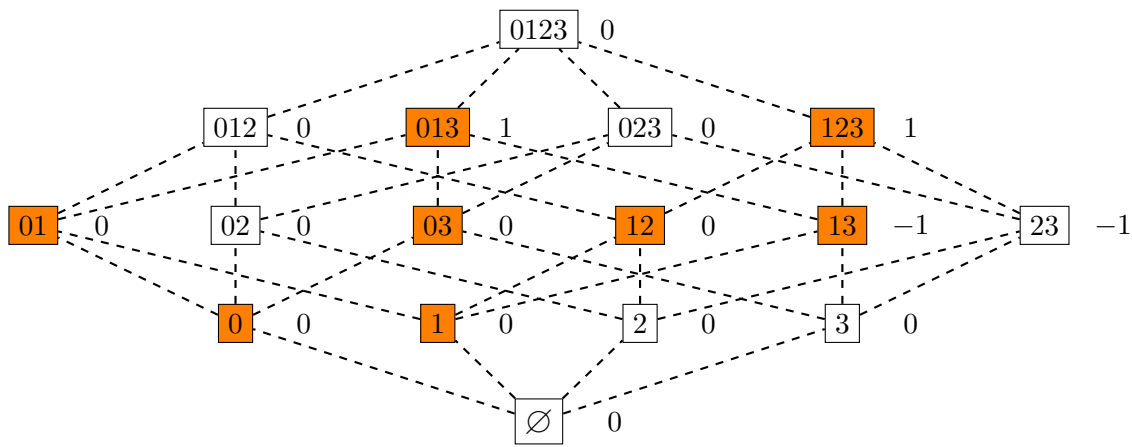
\begin{figure}
\begin{center}
\input{figures/0111001001110010.tex}
\caption{Visual representation of the configuration~$\calC$ from
Example~\ref{expl:0111001001110010} and of its associated Möbius
  function~$\muhat_{\bbmB_S,\calC}$ beside each node. Here, $\calC$ consists of all the
colored nodes.}
\label{fig:0111001001110010}
\end{center}
\end{figure}

Next, we define the \emph{non-cancelling principal downsets} of such a
configuration.

\begin{definition}
\label{def:nonzeros}
Let~$S$ finite and~$\calC$ a configuration of~$B_S$. The set of
\emph{non-cancelling principal downsets of~$\bbmB_S$ with respect to~$\calC$},
denoted~$\ncpd_{\bbmB_S}(\calC)$, is defined by
\[\ncpd_{\bbmB_S}(\calC) \defeq \{\frakI_{\bbmB_S}(X) \mid X \in \bbmB_s,\  \muhat_{\bbmB_S,\calC}(X) \neq 0\}.\]
(Recall that we write $\frakI_{\bbmB_S}(X)$ for $\frakI_{\bbmB_S}(\{X\})$.)
\end{definition}

  Note that the elements of $\ncpd_{\bbmB_S}(\calC)$ are subsets of~$\bbmB_S$. The notation might appear heavy at first sight, e.g., using subscripts
  in $\muhat_{\bbmB_S,\calC}$ and $\ncpd_{\bbmB_S}$ to always recall which
  Boolean lattice we are considering. This will however help to avoid confusion
  in later proofs, when multiple Boolean lattices and configurations will be
  involved.

  We are now ready to state the alternative formulation, which we call for convenience the~\emph{$\ncpd$ conjecture}.

\begin{conjecture}[$\ncpd$ conjecture]
\label{conj:ncpd}
For every finite~$S$, for every configuration~$\calI$ of~$\bbmB_{S}$ that is a
downset of~$\bbmB_{S}$, we have that~$\calI \in \bullet(\ncpd_{\bbmB_{S}}(\calI))$.
\end{conjecture}

We point out that, if we do not ask the configuration to be a
downset of~$\bbmB_S$ then the statement is false, i.e., there exists
a finite~$S$ and configuration~$\calC$ of~$\bbmB_S$ such
that~$\calC \notin \bullet(\ncpd_{\bbmB_S}(\calC))$: see
Section~\ref{sec:final}.

We prove in the remaining of this section that this conjecture is
equivalent to the~$\nci$ conjecture. We start with a general-purpose lemma
that gives a correspondence between both formulations:

\begin{lemmarep}
  \label{lem:embeds}
  Let~$S$ be a finite set and let~$\calI$ be a downset of~$\bbmB_S$ of the form~$\calI =
  {\frakI_{\bbmB_S}(\calF)}$ for~$\calF$ a non-trivial set family.
  For~$X\in \calF$ let~$X' \defeq \frakI_{\bbmB_S}(X)$ (which is $2^X$),
  and let~$\calF' \defeq \{X' \mid X \in \calF\}$. Note that 
  $\calF'$ is non-trivial as well, and let then~$L$ be its intersection
  lattice. 
Then we have~$\ncpd_{\bbmB_S}(\calI) = \nci(L)$.
\end{lemmarep}
\begin{proof}
  Observe that~$\onehat_L = \calI$. For~$\calT \subseteq \calF$, we
  write~$\calT'$ for the corresponding subset of~$\calF'$,
  i.e.,~$\calT' = \{X' \mid X\in \calT\}$, and vice versa. For
  $\calT \subseteq \calF, \calT \neq \emptyset$ let~$X_\calT \defeq
  \bigcap \calT$. We show that the following hold:
  \begin{enumerate}
    \item \label{firstitem} For~$\calT'\subseteq \calF', \calT'\neq
      \emptyset$ we have~$S_{\calT'} = {\frakI_{B_S}(X_\calT)}$.
    \item \label{seconditem} Let~$L'$ be the lattice~$(\{X_\calT \mid
      \calT\subseteq \calF, \calT \neq \emptyset\} \cup \{\onehat_{L'}\},
      \leq_{L'})$ where~$\onehat_{L'}$ is a fresh element
      and~$Y\leq_{L'} \onehat_{L'}$ for all $Y\in L'$ and~$X_{\calT_1} \leq_{L'} X_{\calT_2}$
      for~$\calT_1,\calT_2 \subseteq \calF$ if~$X_{\calT_1} \subseteq X_{\calT_2}$. Then~$L
      \simeq L'$ via the isomorphism that sends~$\onehat_L$
      to~$\onehat_{L'}$ and~$S_{\calT'}$ for~$\calT' \subseteq \calF'$, $\calT' \neq \emptyset$ to~$X_T$.
    \item \label{thirditem} For every~$Y\in \bbmB_S$, if~$Y$ is of the form~$X_\calT$ 
      then~$\muhat_{\bbmB_S,\calI}(Y) = - \mu_L(S_{\calT'},\onehat_L)$, otherwise $\muhat_{\bbmB_S,\calI}(Y) = 0$.
  \end{enumerate}

  Notice that Item (\ref{thirditem}) together with Item (\ref{firstitem})
  implies that~$\ncpd_{\bbmB_S}(\calI) = \nci(L)$, which is what we need to
  prove.
  Item (\ref{firstitem}) follows from the chain of trivialities~$S_{\calT'}
  \defeq \bigcap \calT' \defeq \bigcap_{X'\in \calT'} X' = \bigcap_{X\in \calT}
  2^X = 2^{\bigcap \calT} \defeq 2^{X_\calT} \defeq
  {\frakI_{B_S}(X_\calT)}$, and Item (\ref{seconditem}) is easy to
  prove using Item (\ref{firstitem}). 
  For Item (\ref{thirditem}), we show
  that~$\muhat_{\bbmB_S,\calI}$ is as claimed. For~$Y\in \bbmB_S$,
  $Y\notin \calI$, it is clear 
  that~$\muhat_{\bbmB_S,\calI}(Y) = 0$, by definition of $\muhat_{\bbmB_S,\calI}$ and because~$\calI$ is a downset, so we
  can focus on~$Y\in \calI$. We show by top-down induction on~$\calI$ that
  $\muhat_{\bbmB_S,\calI}$ is as claimed on~$\calI$. The base case is
  when~$Y$ is a maximal element of~$\calI$, i.e.,~$Y=X$ for
  some maximal~$X\in \calF$. Then we have~$\muhat_{\bbmB_S,\calI}(Y) = 1 = -
  \mu_L(X')$ indeed.
  The inductive case is when~$Y$ is not a maximal element of~$\calI$.
  Let~$\calT_Y \defeq \{X\in \calF \mid Y \subseteq X\}$, which is non-empty. Notice
  the following fact ($\dagger$): we have $Y \subseteq X_{\calT_Y}$, and more specifically we
  have~$Y\subseteq X_\calT$ for~$\calT\subseteq \calF, \calT\neq \emptyset$ if
  and only if~$\calT \subseteq \calT_Y$. We then distinguish two cases.
  \begin{itemize}
    \item The first case is when~$Y=X_{\calT_Y}$. We
      have $\muhat_{\bbmB_S,\calI}(Y) = 1 - \sum_{Y \subsetneq Y'}
      \muhat_{\bbmB_S,\calI}(Y')$ by definition. Now by what precedes, if $Y'
      \notin \calI$ then $\muhat_{\bbmB_S,\calI}(Y') = 0$, and by
      induction hypothesis if $Y'$ is not of the form $X_\calT$
      then $\muhat_{\bbmB_S,\calI}(Y') = 0$. Hence, we have
      $\muhat_{\bbmB_S,\calI}(Y) = 1 - \sum_{\substack{Y \subsetneq
      Y'\\ Y' = X_{\calT}}} \muhat_{\bbmB_S,\calI}(X_{\calT})$, and
      the suitable $\calT$ are the subsets of~$\calT_Y$ by
      ($\dagger$), i.e., $\muhat_{\bbmB_S,\calI}(Y) = 1 - \sum_{Y'
      \in \{X_{\calT} \mid \calT \subseteq \calT_Y\} \setminus
    \{Y\}} \muhat_{\bbmB_S,\calI}(X_{\calT})$. Now, by induction
    hypothesis, we have $\muhat_{\bbmB_S,\calI}(X_{\calT}) = -
    \mu_L(S_{\calT'},\onehat_L)$ for all $\calT \subseteq \calT_Y$
    such that~$X_{\calT} \neq Y$, therefore
    $\muhat_{\bbmB_S,\calI}(Y) = 1 + \sum_{Y' \in \{X_{\calT} \mid
    \calT \subseteq \calT_Y\} \setminus \{Y\}}
    \mu_L(S_{\calT'},\onehat_L)$. Now by Item (\ref{seconditem})
    this is equal to $1 + \sum_{\substack{U \in L\\U\neq \onehat_L
    \\ S_{\calT_Y'} \subsetneq U}} \mu_L(U,\onehat_L)$, which is
    equal to $\sum_{\substack{U \in L\\ S_{\calT_Y'} \subsetneq U}}
    \mu_L(U,\onehat_L)$, which is $-\mu_L(S_{\calT_Y'})$ by
    definition. So indeed $\muhat_{\bbmB_S,\calI}(Y) = -
    \mu_L(S_{\calT_Y'})$.

    \item The second case is when~$Y \neq X_{\calT_Y}$. Then, the
      nodes of $\calI$ in the principal upset of~$Y$ minus~$\{Y\}$ consists of
      the nodes of~$\calI$ in the principal upset of
      $X_{\calT_Y}$, and of nodes that are not of the form
      $X_\calT$ by ($\dagger$) and that are strict supersets of~$Y$. By induction hypothesis, the
      $\muhat_{\bbmB_S,\calI}$-value of the latter is zero. So
      $\muhat_{\bbmB_S,\calI}(Y) = 1 - \sum_{X_{\calT_Y} \subseteq
      Y'} \muhat_{\bbmB_S,\calI}(Y') = 1 -
      \muhat_{\bbmB_S,\calI}(X_{\calT_Y}) - \sum_{X_{\calT_Y}
        \subsetneq Y'} \muhat_{\bbmB_S,\calI}(Y')$, but by
        definition of $\muhat_{\bbmB_S,\calI}$ we have
        $\muhat_{\bbmB_S,\calI}(X_{\calT_Y}) = 1 -
        \sum_{X_{\calT_Y} \subsetneq Y'}
        \muhat_{\bbmB_S,\calI}(Y')$, so indeed
        $\muhat_{\bbmB_S,\calI}(Y) = 0$. This concludes the proof.\qedhere
  \end{itemize}
\end{proof}

The proof of this lemma is tedious and we defer it to
Appendix~\ref{apx:alt}. We then show each direction of the
equivalence in turn.

\subparagraph*{$\nci$ conjecture (formulation I) $\implies$ $\ncpd$ conjecture.}
Assume Conjecture~\ref{conj:general} to be true, let~$\calI$ be a
configuration of~$\bbmB_{S}$ for some finite set~$S$, assume that $\calI$ is a downset, and let us show that~$\calI \in
\bullet(\ncpd_{\bbmB_{S}}(\calI))$. If~$\calI$ is empty this is clear because then we have 
$\bullet(\ncpd_{\bbmB_{S}}(\calI)) = \bullet(\emptyset) = \{\emptyset\}$.
If~$\calI$ is a principal downset
of~$\bbmB_S$ this is also clear, because in that case we have
that $\muhat_{\bbmB_S,\calI}(X)$ for~$X\in \bbmB_S$ equals~$1$ if~$\calI = \frakI_{\bbmB_S}(X)$ and~$0$ otherwise,
so assume that~$\calI$ is not empty and is not a principal
downset. Then~$\calI$ is of the form $\frakI_{\bbmB_S}(\calF)$ for some non-trivial set family~$\calF$.  
Let then~$L$ be the intersection lattice constructed from~$S$ and~$\calF$
as in Lemma~\ref{lem:embeds}. By this Lemma we have
that $\ncpd_{\bbmB_S}(\calI) = \nci(L)$. By the hypothesis of
Conjecture~\ref{conj:general} being true we have~$\onehat_L \in
\bullet(\nci(L))$, but~$\onehat_L$ is~$\calI$ so indeed $\calI \in
\bullet(\ncpd_{\bbmB_{S}}(\calI))$.\\

\subparagraph*{$\ncpd$ conjecture $\implies$ $\nci$ conjecture (formulation II).}
Assume now that Conjecture~\ref{conj:ncpd} is true, let~$\calF$ be a
non-trivial set family 
such that the intersection lattice~$L = \bbmL_\calF$ is tight, 
and let us show that~$\onehat_L \in \bullet(\nci(L))$. For~$X\in \calF$ let~$X'
\defeq 2^X$, and let~$\calF'  \defeq \{X' \mid X \in \calF\}$. Note that~$\calF'$ is non-trivial as well, so let us consider~$L'=\bbmL_{\calF'}$.
For~$\calT \subseteq \calF$, let~$\calT'$ be the corresponding subset of~$\calF'$, i.e.,~$\calT' = \{X' \mid X \in \calT\}$, and vice versa.
Then observe that
for~$\calT\subseteq \calF, \calT \neq \emptyset$ we have~$S_{\calT'} =
2^{S_\calT}$, so that~$L\simeq L'$.
Moreover we have that~$L'$ is full: indeed it is clear that
for~$\calT\subseteq \calF, \calT\neq \emptyset$ we have $S_\calT
\in \{X \in \onehat_{L'} \mid \min_{L'}(X) =
S_{\calT'}\}$.\footnote{But note that~$L'$ is never tight: indeed
we have, e.g., $\{X \in \onehat_{L'} \mid \min_{L'}(X) =
\zerohat_{L'}\} = \{\zerohat_L, \emptyset\}$.}
Define now~$S= \onehat_L$ (which is finite because~$L$ is tight), and let~$\calI$ be the downset of~$\bbmB_S$ that is
$\frakI_{\bbmB_S}(\calF)$. Observe then that~$L'$ is exactly the
intersection lattice that we would obtain from the construction described in
Lemma~\ref{lem:embeds} applied on~$S$ and~$\calF$, so that~$\ncpd_{\bbmB_S}(\calI) =
\nci(L')$ by that lemma. Since Conjecture~\ref{conj:ncpd} is true by hypothesis we have~$\calI
\in \bullet(\ncpd_{\bbmB_S}(\calI))$, hence~$\calI \in \bullet(\nci(L'))$. But $\calI =
\{Y\in \bbmB_S \mid \exists X \in \calF,~Y\subseteq X\}$ by definition, which
is equal to $\bigcup \calF'$, which
is~$\onehat_{L'}$ by definition. Therefore~$\onehat_{L'} \in
\bullet(\nci(L'))$. 
This implies~$\onehat_L \in \bullet(\nci(L))$ as well by
Lemma~\ref{lem:from-full}.

\subsection{Useful facts}
\label{subsec:useful}
We now present useful facts about the $\ncpd$ formulation of the
conjecture that motivate its introduction.

First, we observe that the generalized Möbius functions are linear (in their parameters)
under disjoint union and subset complement.

\begin{fact}
\label{fact:gen-mob-linear}
    Let~$S$ finite and~$\calC_1,\calC_2$ be two configurations of~$\bbmB_S$ such
that~$\calC_1 \cupdot \calC_2$ (resp., $\calC_1 \minusdot \calC_2$) is well defined. Then
$\muhat_{\bbmB_S,\calC_1 \cupdot \calC_2} = \muhat_{\bbmB_S,\calC_1 } + \muhat_{\bbmB_S,\calC_2}$ (resp.,
$\muhat_{\bbmB_S,\calC_1 \minusdot \calC_2} = \muhat_{\bbmB_S,\calC_1} - \muhat_{\bbmB_S,\calC_2}$).
\end{fact}
\begin{proof}
    One can simply check by top-down induction that for all $X\in \bbmB_S$ we
have $\muhat_{\bbmB_S,\calC_1 \cupdot \calC_2}(X) = \muhat_{\bbmB_S,\calC_1}(X) +
\muhat_{\bbmB_S,\calC_2}(X)$ (resp., $\muhat_{\bbmB_S,\calC_1 \minusdot \calC_2}(X) =
\muhat_{\bbmB_S,\calC_1}(X) - \muhat_{\bbmB_S,\calC_2}(X)$).
\end{proof}

Second, we show that, if we allow any principal downset in the dot algebra, then we can
build any configuration: this is the analogue of
Fact~\ref{fact:annoyingfact}. In fact we will use the following
stronger result:\footnote{In this lemma and its proof we
correctly use the notation $\frakI_{\bbmB_S}(\{X\})$ for~$X\in
\bbmB_S$ instead of $\frakI_{\bbmB_S}(X)$ to avoid confusion.}

\begin{lemma}
\label{lem:allreach}
Let~$\calC$ be a configuration of~$\bbmB_S$ (for $S$ finite), and let
  \[A_\calC \defeq \{\frakI_{\bbmB_S}(\{X\}) \mid X\in
  {\frakI_{\bbmB_S}(\calC)} \}.\]
Then we have~$\calC \in \bullet(A_\calC)$.
\end{lemma}
\begin{proof}
By induction on the size of~$\frakI_{\bbmB_S}(\calC)$.
 The base case is when~$|\frakI_{\bbmB_S}(\calC)| = 0$, which
 can only happen when
 $\calC$ is $\emptyset$.
  Then clearly $\calC \in \bullet(A_\calC)$ as $A_\calC = \emptyset$ and $\bullet(\emptyset) = \{\emptyset\}$.

For the inductive case, assume the
claim is true for all~${\calC'}$ with~$|\frakI_{\bbmB_S}({\calC'})| \leq n$, and let us
show it is true for~$\calC$ with $|\frakI_{\bbmB_S}(\calC)| = n+1$. Let~$X \in
\calC$ be maximal, and consider~${\calC'} \defeq \calC \setminus
\{X\}$.  Then $|\frakI_{\bbmB_S}({\calC'})| \leq n$, hence by induction hypothesis 
we have $\calC' \in \bullet(A_{\calC'})$.
But it is clear that~$A_{{\calC'}} \subseteq A_\calC$, so that
$\calC' \in \bullet(A_{\calC})$ as well.
Moreover,
  consider~$\calC_1 \defeq {\frakI_{\bbmB_S}(\{X\}) \setminus \{X\}}$. By induction hypothesis and for the same reasons 
we have~$\calC_1 \in \bullet(A_{\calC})$.
Then~$\calC' \cupdot (\frakI_{\bbmB_S}(\{X\}) \minusdot \calC_1)$ is a valid expression
witnessing that~$\calC \in \bullet(A_{\calC})$.
\end{proof}

Third, we show that all instances of the~$\ncpd$ conjecture are
already “full”, in the sense that an analogous version of
Proposition~\ref{prp:full-multiplicity} automatically holds for them
(whereas we saw that Proposition~\ref{prp:full-multiplicity} might
be false for intersection lattices that are not full). 
 Let~$\pd(\bbmB_S)$ be the set of principal downsets of~$\bbmB_S$, i.e.,~$\pd(\bbmB_S) \defeq \{ \frakI_{\bbmB_S}(X) \mid X \in \bbmB_S\}$. We show:

\begin{proposition}
  \label{prp:ncpdfull}
   Let~$S$ finite and~$\calI$ be a configuration of~$\bbmB_S$ that
   is a downset and~$T$ a tree witnessing that $\calI \in \bullet(\pd(\bbmB_S))$. Then we have $\mult_T(\frakI_{\bbmB_S}(X)) = 
   \muhat_{\bbmB_S,\calI}(X)$ for every~$X \in \bbmB_S$.
\end{proposition}
\begin{proof}
  It is routine to show by (bottom-up) induction on~$T$ that, for every
  node~$n\in T$, letting~$\calC_n$ be the corresponding
  configuration represented by the subtree of~$T$ rooted at~$n$ we
  have that $\mult_{T_n}(\frakI_{\bbmB_S}(X)) =
  \muhat_{\bbmB_S,\calC_n}(X)$ for every~$X \in \bbmB_S$; this in particular uses the linearity of
  generalized Möbius functions under disjoint unions and subset complements. Applying
  this claim to the root of~$T$ gives the claim.
\end{proof}

Last, we establish a connection between the generalized Möbius
function of a configuration~$\calC$ and the \emph{Euler
characteristic} of~$\calC$, that will be crucial in the next
section.

\begin{definition}
  Let~$\calC$ be a set family of finite sets. We define
the~\emph{Euler characteristic of~$\calC$}, denoted~$\eul(\calC)$,
by $\eul(\calC) \defeq \sum_{X\in \calC} (-1)^{|X|}$.
\end{definition}

\begin{proposition}
  \label{prp:mob-eul}
  Let~$S$ finite and $\calC$ a configuration of~$\bbmB_S$. Then we have $\muhat_{\bbmB_S,\calC}(X) = (-1)^{|X|} \times \eul(\calC \cap \frakF_{\bbmB_S}(X))$ for every~$X\in \bbmB_S$.
\end{proposition}
\begin{proof}
  Define~$f,g:\bbmB_S \to \ZZ$ by~$f(X) \defeq
  \muhat_{\bbmB_S,\calC}(X)$ and~$g(X) = 1$ if~$X\in \calC$ and~$0$
  otherwise. Observe then that, by definition
  of~$\muhat_{\bbmB_S,\calC}$ we have $g(X) = \sum_{\substack{X'\in
  \bbmB_S \\ X\subseteq X'}} f(X')$ for all $X\in \bbmB_S$.
  Then by Proposition~\ref{prp:inv_bool} we have that
  $f(X) =
    \sum_{\substack{X'\in \bbmB_S \\ X\subseteq X'}}
    (-1)^{|X'|-|X|} g(X')$ for all $X\in \bbmB_S$, i.e.,
    \begin{align*}
      \muhat_{\bbmB_S,\calC}(X) &= \sum_{\substack{X'\in \bbmB_S \\ X\subseteq X'}}
    (-1)^{|X'|-|X|} g(X')\\
      &= \sum_{X'\in {\calC \cap \frakF_{\bbmB_S}(X)} }
    (-1)^{|X'|-|X|}\\
&= (-1)^{|X|} \times \sum_{X'\in \calC \cap \frakF_{\bbmB_S}(X) }
    (-1)^{|X'|}\\
&= (-1)^{|X|} \times \eul(\calC \cap \frakF_{\bbmB_S}(X))
    \end{align*}
    for all~$X\in \bbmB_S$, which is what we wanted to show.
\end{proof}

Connections between the Möbius function and the Euler
characteristic have already been shown, for instance see Philip
Hall's theorem~\cite[Proposition 6 and Theorem
3]{rota1964foundations}. As far as we can tell, however, the
connection from Proposition~\ref{prp:mob-eul} seems new.

%% file: figures/0111001001110010.tex
\begin{tikzpicture}[xscale=1.4]
\tikzset{nodestyle/.style={draw,rectangle}}
 ===== NODES ====

  \node[nodestyle] (emptyset) at (0.0, 0.0) {$\emptyset$};
  \node[right of=emptyset] {$0$};
  \node[nodestyle,fill=orange] (0) at (-2.85, 1.3) {$0$};
  \node[right of=0,xshift=-.1cm] {$0$};
  \node[nodestyle,fill=orange] (1) at (-0.95, 1.3) {$1$};
  \node[right of=1,xshift=-.1cm] {$0$};
  \node[nodestyle] (2) at (0.95, 1.3) {$2$};
  \node[right of=2,xshift=-.1cm] {$0$};
  \node[nodestyle] (3) at (2.85, 1.3) {$3$};
  \node[right of=3,xshift=-.1cm] {$0$};
  \node[nodestyle,fill=orange] (01) at (-4.75, 2.6) {$01$};
  \node[right of=01,xshift=-.1cm] {$0$};
  \node[nodestyle] (02) at (-2.85, 2.6) {$02$};
  \node[right of=02,xshift=-.1cm] {$0$};
  \node[nodestyle,fill=orange] (03) at (-0.95, 2.6) {$03$};
  \node[right of=03,xshift=-.1cm] {$0$};
  \node[nodestyle,fill=orange] (12) at (0.95, 2.6) {$12$};
  \node[right of=12,xshift=-.1cm] {$0$};
  \node[nodestyle,fill=orange] (13) at (2.85, 2.6) {$13$};
  \node[right of=13,xshift=-.1cm] {$-1$};
  \node[nodestyle] (23) at (4.75, 2.6) {$23$};
  \node[right of=23,xshift=-.1cm] {$-1$};
  \node[nodestyle] (012) at (-2.85, 3.9) {$012$};
  \node[right of=012,xshift=-.1cm] {$0$};
  \node[nodestyle,fill=orange] (013) at (-0.95, 3.9) {$013$};
  \node[right of=013,xshift=-.1cm] {$1$};
  \node[nodestyle] (023) at (0.95, 3.9) {$023$};
  \node[right of=023,xshift=-.1cm] {$0$};
  \node[nodestyle,fill=orange] (123) at (2.85, 3.9) {$123$};
  \node[right of=123,xshift=-.1cm] {$1$};
  \node[nodestyle] (0123) at (0.0, 5.2) {$0123$};
  \node[right of=0123,xshift=-.1cm] {$0$};

 ===== EDGES ====

\draw[black,thick,dashed] (emptyset) -- (0);
\draw[black,thick,dashed] (emptyset) -- (1);
\draw[black,thick,dashed] (emptyset) -- (2);
\draw[black,thick,dashed] (emptyset) -- (3);
\draw[black,thick,dashed] (0) -- (01);
\draw[black,thick,dashed] (0) -- (02);
\draw[black,thick,dashed] (0) -- (03);
\draw[black,thick,dashed] (1) -- (01);
\draw[black,thick,dashed] (1) -- (12);
\draw[black,thick,dashed] (1) -- (13);
\draw[black,thick,dashed] (2) -- (02);
\draw[black,thick,dashed] (2) -- (12);
\draw[black,thick,dashed] (2) -- (23);
\draw[black,thick,dashed] (3) -- (03);
\draw[black,thick,dashed] (3) -- (13);
\draw[black,thick,dashed] (3) -- (23);
\draw[black,thick,dashed] (01) -- (012);
\draw[black,thick,dashed] (01) -- (013);
\draw[black,thick,dashed] (02) -- (012);
\draw[black,thick,dashed] (02) -- (023);
\draw[black,thick,dashed] (03) -- (013);
\draw[black,thick,dashed] (03) -- (023);
\draw[black,thick,dashed] (12) -- (012);
\draw[black,thick,dashed] (12) -- (123);
\draw[black,thick,dashed] (13) -- (013);
\draw[black,thick,dashed] (13) -- (123);
\draw[black,thick,dashed] (23) -- (023);
\draw[black,thick,dashed] (23) -- (123);
\draw[black,thick,dashed] (012) -- (0123);
\draw[black,thick,dashed] (013) -- (0123);
\draw[black,thick,dashed] (023) -- (0123);
\draw[black,thick,dashed] (123) -- (0123);
\end{tikzpicture}

%% file: partial.tex
In this section we present a partial result on the $\ncpd$ conjecture
intuitively saying that we can “avoid” any given non-trivial zero. We first
define what these are. 

\begin{definition}
  \label{def:ntcz}
  Let~$\calI$ be a downset of~$\bbmB_S$. We define the \emph{non-trivial zeros
  of~$\calI$}, denoted~$\ntz_{\bbmB_S}(\calI)$, by
  \[\ntz_{\bbmB_S}(\calI) \defeq \{Z \in \calI \mid \muhat_{\bbmB_S,\calI}(Z) = 0 \}.\]
    By opposition, the \emph{trivial} zeros of~$\calI$ are the elements~$Z \in
    \bbmB_S \setminus \calI$ such that~$\muhat_{\bbmB_S,\calI}(Z) = 0$ (notice
    that these form an upset of~$\bbmB_S$).
\end{definition}

Observe that if $\calI$ does not have non-trivial zeros, i.e., $\ntz_{\bbmB_S}(\calI) = \emptyset$, then
we have indeed that $\calI \in \bullet(\ncpd_{\bbmB_{S}}(\calI))$, that is, the $\ncpd$ conjecture holds, as can be
seen by taking~$\calC=\calI$ in Lemma~\ref{lem:allreach}.
So we can focus on downsets~$\calI$ having as least one such non-trivial zero.

The partial result that we show in this section is that, for any given such non-trivial zero~$Z$, if we are allowed
to use all the non-trivial principal downsets \emph{except the one generated by~$Z$}, then we can construct~$\calI$.
Formally we will prove the following:

\begin{theorem}
  \label{thm:partial}
  Let~$\calI$ be a downset of~$\bbmB_S$ such that $\ntz_{\bbmB_S}(\calI)$ is not empty, and let~$Z\in \ntz_{\bbmB_S}(\calI)$.
  Then we have~$\calI \in \bullet(\{\downarrow_{\bbmB_S}(X) \mid X\in I\setminus \{Z\}\})$.
\end{theorem}

The proof reuses and extends some ideas and
notions from~\cite{monet2020solving}. We first define some of these notions
in Section~\ref{subsec:pairs},
in particular the notion of \emph{adjacent
pairs} and certain equivalence classes. In
Section~\ref{subsec:lifting-and-simul} we show what we call the \emph{lifting lemma}, and show adjacent pairs can be
simulated with principal downsets.
We combine everything in Section~\ref{subsec:proof} to prove Theorem~\ref{thm:partial}.

\subsection{Adjacent pairs}
  \label{subsec:pairs}

\begin{definition}
\label{def:pair}
An \emph{adjacent pair} of~$\bbmB_S$ is a configuration~$\calP$ of
the form~$\{X,X\minusdot \{x\}\}$ (with $X \in \bbmB_S$ and~$x\in
X$).
\end{definition}

\begin{definition}
\label{def:pairs_game}
Let~$A$ be a set of adjacent pairs of~$\bbmB_S$, 
  and let~$\calC,\calC'$ be two configurations 
of~$\bbmB_S$.  We then write $\calC \rewr{+ A}{3} \calC'$ when there
exists an adjacent pair~$\calP\in A$ such that
$\calC' = \calC \cupdot \calP$.  
Similarly we write
$\calC \rewr{-A}{3} \calC'$ whenever $\calC' \rewr{+A}{3}
\calC$ (i.e., when there exists~$\calP\in A$ such that~$\calC' = \calC \minusdot \calP$), and 
write $\calC \rewr{$\pm$ A}{3} \calC'$ when
$\calC \rewr{+A}{3} \calC'$ or $\calC \rewr{-A}{3}
\calC'$.  Observe that $\rewr{$\pm A$}{3}$ is symmetric.  We write
$\rewr{$\pm A$}{3}^*$ the reflexive transitive closure of $\rewr{$\pm A$}{3}$,
and write~$\simeq_{A}$ the induced equivalence relation.
\end{definition}

In other words, we have $\calC \simeq_{A} \calC'$ when we can go
from~$\calC$ to~$\calC'$ by iteratively (1) adding an 
adjacent pair from~$A$ to the current configuration if the pair is disjoint from the
configuration; or (2) removing an adjacent pair from~$A$ to the current
configuration if the pair was included in the configuration.

\begin{definition}
  \label{def:bla}
  For a subset~$\calG$ of~$\bbmB_S$, define the set of 
  \emph{allowed adjacent pairs of~$\bbmB_S$ relative to~$\calG$},
  denoted~$\ap_{\bbmB_S}(\calG)$, by
  \[\ap_{\bbmB_S}(\calG) \defeq \{\calP \mid \calP \text { is an adjacent pair of } \bbmB_S \text { and } \calP \subseteq \calG\}.\]
\end{definition}

\begin{definition}
  \label{def:connected}
  For~$S$ finite, define the undirected graph $G_{S}$ with
  vertex set~$2^S$ whose edges are all adjacent pairs of~$\bbmB_S$.
  A subset~$\calG$ of~$\bbmB_S$ is \emph{connected} if it is
  a connected set of nodes in $G_S$.
\end{definition}

The goal of this section is to establish the following proposition:

\begin{proposition}
  \label{prp:eul-equiv}
  Let~$\calG$ be a connected subset of~$\bbmB_S$, and
  $\calC_1, \calC_2$ two configurations that are included in~$\calG$. We have $\eul(\calC_1) = \eul(\calC_2)$ if and only if
  $\calC_1 \simeq_{\ap_{\bbmB_S}(\calG)} \calC_2$.
\end{proposition}

This result already appears in \cite[Proposition
6.1]{monet2020solving}, but only when~$\calG$ is the whole Boolean
lattice, i.e., when\footnote{We will actually only use this result
when~$\calG$ is a downset of~$\bbmB_S$, but we still prove this more
general version for completeness and because it is not much more complicated.} 
we have~$\calG = 2^S$.
Observe that the “if” direction in Proposition~\ref{prp:eul-equiv} is trivial, since adding or removing an adjacent pair (with disjoint union or subset complement) does not modify the Euler characteristic. Hence we need to prove the “only if” direction.
To this end, we reproduce the following lemma from~\cite{monet2020solving}:

\begin{lemma}[Lemma 5.10 of \cite{monet2020solving}]
\label{lem:chaining}
	Let~$\calC$ be a configuration of~$\bbmB_S$, and $X \neq X'$ be two subsets of~$S$ such that there is a simple path~$P$ of the form
	$X = X_0 - \cdots - X_{n+1} = X'$ from $X$ to $X'$
	in~$G_S$ with $n \geq 0$ and $X_i \notin \calC$ for $1 \leq i \leq n$.
	Then we have the following:
	\begin{description}
          \item[\textbf{Erasing.}] If $(-1)^{|X|} \neq (-1)^{|X'|}$ (i.e., $n$ is even) and $\{X,X'\} \subseteq \calC$ then,
			defining~$\calC'$ by $\calC' \defeq \calC \setminus \{X,X'\}$, we have $\calC \rewr{$\pm \ap_{\bbmB_S}(P)$}{8}^* \calC'$.
			We say that we go from~$\calC$ to~$\calC'$ by \emph{erasing~$X$ and~$X'$}.
                      \item[\textbf{Teleporting.}] If $(-1)^{|X|} = (-1)^{|X'|}$ (i.e., $n$ is odd) and $X \in \calC$ and $X' \notin \calC$ then,
			defining~$\calC'$ by $\calC' \defeq (\calC \setminus \{X\}) \cup \{X'\}$, we have $\calC \rewr{$\pm \ap_{\bbmB_S}(P)$}{8}^* \calC'$.
			We say that we go from~$\calC$ to~$\calC'$ by \emph{teleporting~$X$ to~$X'$}.
	\end{description}
\end{lemma}
\begin{proof}
        We only explain the erasing part, as teleporting works
        similarly. Let $n=2i$. For $0 \leq j < i$, do the
        following: add the adjacent pair $\{X_{2j+1}, X_{2j+2}\}$
        and remove the adjacent pair $\{X_{2j},X_{2j+1}\}$.
        Finally, remove the adjacent pair~$\{X_{2i}, X_{2i+1}\}$.
\end{proof}

Teleporting is illustrated in Figure~\ref{fig:teleporting} (adapted from~\cite{monet2020solving}). 

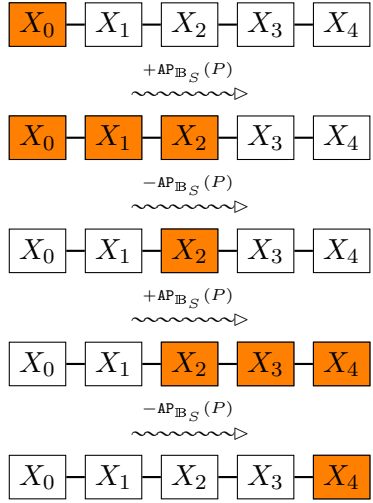
\begin{figure}
\centering
\input{figures/chainswap}
\caption{Imagine that the path at the top occurs in $G_S$ for some configuration~$\calC$ of~$\bbmB_S$ (in orange), and let~$P = \{X_0,\ldots,X_4\}$.
	The consecutive orange configurations of~$\bbmB_S$ are obtained by single steps of the transformation.
      The total transformation illustrates what is called \emph{teleporting} in Lemma~\ref{lem:chaining}: we go from~$\calC$ to~$\calC'$ in the bottom path by teleporting $X_1$ to~$X_4$.}
\label{fig:teleporting}
\end{figure}

We will also reuse the \emph{fetching lemma} from~\cite{monet2020solving},
extending it to connected subsets of~$\bbmB_S$ (instead of
the full Boolean lattice).
Given the right circumstances, this lemma fetches for us two sets~$X,X'\in
\calC$ and a suitable path so that we can erase~$X$ and $X'$. Formally:

\begin{lemma}[Fetching lemma, slightly extending Lemma 5.11 of \cite{monet2020solving}]
\label{lem:fetch}
	Let~$\calG$ be a subset of~$\bbmB_S$ and~$\calC \subseteq \calG$ a configuration such that $|\calC| \neq |\eul(\calC)|$.
	Then there exist $X,X'\in \calC$ with $(-1)^{|X|} \neq (-1)^{|X'|}$ and a simple path
	$X = X_0 - \cdots - X_{n+1} = X'$ from $X$ to $X'$ in 
	$G_S$ (hence with $n$ even) with all nodes in this path being in~$\calG$ such that $X_i \notin \calC$ for $1 \leq i \leq n$.
\end{lemma}
\begin{proof}
	Since $|\calC| \neq |\eul(\calC)|$, there exist $X'',X''' \in \calC$ with $(-1)^{|X''|} \neq (-1)^{|X'''|}$.
	Let $X'' = X''_0 - \cdots - X''_{m+1} = X'''$ be an arbitrary simple path from $X''$ to $X'''$ in 
	$G_S$ with all nodes being in~$\calG$ (such a path clearly exists because $\calG$ is a connected subset of~$G_S$).
	Now, let $k_1 \defeq \max(0 \leq j \leq m \mid (-1)^{|X''_j|} = (-1)^{|X''|} \text{ and }X''_j \in \calC)$, and then
	let $k_2 \defeq \min(k_1 < j \leq m+1 \mid (-1)^{|X''_{j}|} = (-1)^{|X'''|} \text{ and }X''_j \in \calC)$, which are well-defined.
        Then we can take $X$ to be $X''_{k_1}$ and $X'$ to be $X''_{k_2}$, which satisfy the desired property.
\end{proof}

With these in place we are ready to prove
Proposition~\ref{prp:eul-equiv}.

\begin{proof}[Proof of Proposition~\ref{prp:eul-equiv}]
  As mentioned above, we only need to prove the “only if” part of the statement.
  Starting from~$\calC_1$, we repeatedly use the fetching lemma and
  Erasing until we obtain a configuration~$\calC_1' \subseteq
  \calG$ such that $\calC_1' \simeq_{\ap_{\bbmB_S}(\calG)} \calC_1$
  and~$|\calC_1'| = \eul(\calC_1')$. We do the same with~$\calC_2$,
  obtaining~$\calC_2' \subseteq \calG$ with $\calC_2'
  \simeq_{\ap_{\bbmB_S}(\calG)} \calC_2$ and~$|\calC_2'| =
  \eul(\calC_2')$. If~$\calC_1' = \calC_2'$ we are done.
  Otherwise, let~$n \defeq |\calC_1'
  \setminus \calC_2'|$, which is~$>0$ because~$|\calC_1'| = |\calC_2'|$. We build by induction a
  sequence of configurations~$\calC_1' = \calC''_0,\ldots,\calC''_n
  = \calC_2'$ that are all included in~$\calG$ and that satisfy (1)  $\calC''_i
  \simeq_{\ap_{\bbmB_S}(\calG)} \calC''_{i+1}$ for all~$0 \leq i
  \leq n-1$ and (2) $|\calC''_i| = \eul(\calC''_i)$ and~$|\calC''_i \setminus \calC_2'| =
  n - i $ for all~$0 \leq i \leq n$. It is clear that this indeed
  implies the claim. As~$\calC''_0$ satisfies condition (2), it is enough to explain how to build $\calC''_{i+1}$
  from~$\calC''_i$ for~$i< n$, and then to argue that the~$\calC''_n$ from this construction is indeed equal to~$\calC_2'$. We have $|\calC''_i \setminus \calC_2'| =
  n - i > 1$ by induction hypothesis, so there exists~$X'\in \calC''_i \setminus \calC_2'$ and~$X'' \in \calC_2' \setminus \calC''_i$ (because~$|\calC_i''|  = \eul(\calC_i'') = \eul(\calC_1') = |\calC_1'| = |\calC_2'|$).
	Let $X' = X'_0 - \cdots - X'_{m+1} = X''$ be an arbitrary simple path from $X'$ to $X''$ in 
	$G_S$ with all nodes being in~$\calG$ (such a path exists because $\calG$ is a connected subset of~$G_S$, and~$m$ is odd and~$\geq 1$).
        Observe that all nodes~$X'_k$ with~$k$ odd are not in~$\calC''_i \cup
        \calC_2'$, because $|\calC''_i| = \eul(\calC''_i)|$ and
        $|\calC_2'| = \eul(\calC_2')$ .
	Now, let $k_1 \defeq \max(0 \leq j \leq m \mid X'_j \in \calC''_i \setminus \calC_2')$, and then
        let $k_2 \defeq \min(k_1 < j \leq m+1 \mid X'_j \in \calC_2' \setminus \calC''_i)$, which are well-defined. Define~$\calC''_{i+1} \defeq \calC''_i \setminus \{X_{k_1}\} \cup \{X_{k_2}\}$. 
        Notice that~$\calC''_{i+1}$ satisfies condition (2), so all we need
        to show is that~$\calC''_i \simeq_{\ap_{\bbmB_S}(\calG)}
        \calC''_{i+1}$.
        We know the following: for a node~$X'_k$ with~$k_1 < k <
        k_2$, if~$k$ is odd then~$X'_k \notin \calC''_i \cup
        \calC'_2$, and if~$k$ is even then either~$X'_k \in
        \calC''_i \cap \calC'_2$ or $X'_k \notin \calC''_i \cup
        \calC'_2$. If there is no node~$X'_k$ with~$k_1 < k < k_2$,
        $k$ even such that~$X'_k \in \calC_i''$, then we can simply
        use Lemma~\ref{lem:chaining} to teleport~$X_{k_1}$
        to~$X_{k_2}$ and we are done. Otherwise,
        let~$X'_{\ell_1},\ldots,X'_{\ell_m}$ with~$m\geq 1$ be all
        the nodes with~$k_1 <\ell_p < k_2$ even that are
        in~$\calC''_i$, in order, i.e.,~$\ell_1 < \ldots< \ell_m$.
        Then we can successively teleport~$X'_{\ell_{m}}$ to
        $X'_{k_2}$, $X'_{\ell_{m-1}}$ to $X'_{\ell_{m}}$, and so
        on, until we teleport~$X'_{\ell_1}$ to $X'_{\ell_2}$ and
        finally we teleport~$X'_{k_1}$ to~$X'_{\ell_1}$, thus
        obtaining~$\calC''_{i+1}$ as promised. 
        It is then clear that~$\calC''_n$ is $\calC_2'$, because we have~$|\calC''_n \setminus \calC'_2| = 0$ by condition (2), and~$|\calC''_n| = |\calC'_2|$ by construction.
        This concludes the
        proof.
\end{proof}

\subsection{Lifting lemma and simulating adjacent pairs with principal downsets}
  \label{subsec:lifting-and-simul}

  A simple, yet important component of the proof of Theorem~\ref{thm:partial} will be what we
  call the \emph{lifting lemma}. We first state and prove it before
  giving intuition.

  \begin{lemma}[Lifting lemma]
    \label{lem:lifting}
    Let~$Z \in \bbmB_S$ and $\calA \subseteq {\frakF_{\bbmB_S}(Z)}$. Let
    $I_\calA \defeq \{{\frakI_{\bbmB_S}(X)} \cap {\frakF_{\bbmB_S}(Z)} \mid X \in \calA\}$
    and $I'_\calA \defeq \{\frakI_{\bbmB_S}(X)\mid X \in \calA\}$. Then, for any
    $\calC \in \bullet(I_\calA)$, defining the
    configuration
    \[\lift_{\bbmB_S,Z}(\calC) \defeq \{X \in \bbmB_S \mid X \cup Z \in \calC\},\]
    we have that~$\lift_{\bbmB_S,Z}(\calC) \in \bullet(I'_\calA)$.
  \end{lemma}
  \begin{proof}
    This is shown by induction on~$\calC \in \bullet(I_\calA)$. The
    first base case is when~$\calC = \emptyset$, but then we
    have~$\lift_{\bbmB_S,Z}(\calC) = \emptyset$ as well, which is
    in~$\bullet(I'_\calA)$ by definition. The second base case is
    when~$\calC = {\frakI_{\bbmB_S}(X)} \cap {\frakF_{\bbmB_S}(Z)}$ for
    some~$X$ in $\calA$. In this case, it is easy to show
    that~$\lift_{\bbmB_S,Z}(\calC) = {\frakI_{\bbmB_S}(X)}$; this
    uses in particular that set union is the join operation of the
    lattice~$\bbmB_S$. But then this implies
    $\lift_{\bbmB_S,Z}(\calC) \in \bullet(I'_\calA)$ indeed.
    For the inductive case, we focus on~$\cupdot$ as~$\minusdot$ works similarly. Let~$\calC_1,\calC_2 \in
    \bullet(I_\calA)$ such that~$\calC \defeq \calC_1 \cupdot
    \calC_2$ is well-defined, and let us show
    that $\lift_{\bbmB_S,Z}(\calC) \in \bullet(\calI'_\calA)$.
    By induction hypothesis we have that
    $\lift_{\bbmB_S,Z}(\calC_1) \in \bullet(\calI'_\calA)$ and
    $\lift_{\bbmB_S,Z}(\calC_2) \in \bullet(\calI'_\calA)$. Now,
    observe that~$\lift_{\bbmB_S,Z}(\calC_1) \cupdot
    \lift_{\bbmB_S,Z}(\calC_2)$ is well defined, and furthermore
    that we have $\lift_{\bbmB_S,Z}(\calC_1 \cupdot \calC_2) =
    \lift_{\bbmB_S,Z}(\calC_1) \cupdot \lift_{\bbmB_S,Z}(\calC_2)$.
    This implies $\lift_{\bbmB_S,Z}(\calC) \in
    \bullet(\calI'_\calA)$ as wanted, and concludes the proof.
  \end{proof}

The intuition behind the lifting lemma is the following. Notice
that~$\frakF_{\bbmB_S}(Z)$ is isomorphic to the Boolean
lattice~$\bbmB_{S'}$ with~$S' \defeq S \setminus Z$. Consider a
witnessing tree~$T'$ representing a configuration~$\calC$ of
$\bbmB_{S'}$ that is formed from principal downsets of~$\bbmB_{S'}$.
Consider the “lifted tree”~$T'$ that is obtained from~$T$ by
replacing every leaf of~$T$ of the form~$\frakI_{\bbmB_{S'}}(X)$
for~$X\subseteq S'$ by the principal downset of~$\bbmB_S$ that is
$\frakI_{\bbmB_{S}}(X\cup Z)$. The lifting lemma says that the
obtained tree~$T'$ is valid (i.e., the internal nodes are well-defined disjoint unions and subset complements) and that it
represents the configuration of~$\bbmB_S$ that
is~$\lift_{\bbmB_S,Z}(\calC)$.

This lemma will have two uses in the proof of Theorem~\ref{thm:partial}. The
first is that said proof will work by first constructing a configuration in 
the principal upset generated by the non-trivial zero $Z$ that is fixed
in the theorem statement, and then this lemma will be used to “lift” the result
to the whole Boolean lattice~$\bbmB_S$. The second use of this lemma is that it
helps us show that adjacent pairs can be simulated by principal downsets, while
avoiding the principal downset that is at the bottom. We explain this 
second use in the following lemma.

\begin{lemma}
  \label{lem:pairs-to-PIs}
  Let~$\calP = \{X, X \minusdot \{x\}\}$ be an adjacent pair
  of~$\bbmB_S$. Then we have~$\calP \in \bullet(\{
  \frakI_{\bbmB_S}(Y)  \mid \emptyset \subsetneq Y \subseteq X\})$.
\end{lemma}

In this statement, pay attention that the principal downset
$\frakI_{\bbmB_S}(\emptyset)$ is not in the base set of the
dot-algebra.\footnote{Recall that to alleviate the notation we write
$\frakI_{\bbmB_S}(\emptyset)$ to mean
$\frakI_{\bbmB_S}(\{\emptyset\})$, which is then~$\{\emptyset\}$.}
In other words, we are allowed to use any principal downset that is
generated by an element that is below~$X$ and strictly
above~$\emptyset$ in~$\bbmB_S$. The reason is that in the proof of
Theorem~\ref{thm:partial}, $Z$ will correspond to~$\emptyset$ here
and it generates the only downset we are not allowed to use. We now
prove Lemma~\ref{lem:pairs-to-PIs}.
\begin{proof}[Proof of Lemma~\ref{lem:pairs-to-PIs}]
  Clearly, it suffices to show the claim in the case that~$X = S$.
  To this end, it is enough to show that we have~$\bbmB_S \minusdot
  \calP \in \bullet(\{ \frakI_{\bbmB_S}(Y)  \mid \emptyset
  \subsetneq Y \subseteq X\})$: indeed we can then obtain~$\calP$
  as $\frakI_{\bbmB_S}(S) \minusdot (\bbmB_S \minusdot \calP)$.
  Let~$\calC'$ be the configuration~$\{X \in \bbmB_S \mid \{x\} \subseteq X \subsetneq S\}$.
  Using Lemma~\ref{lem:allreach} appropriately we have that $\calC' \in
   \bullet(\{
     {\frakI_{\bbmB_{S}}(Y)} \cap {\frakF_{\bbmB_{S}}(\{x\})} \mid \{x\}
\subseteq Y \subsetneq S\})$. We now use the lifting lemma (invoked with~$Z = \{x\}$) to
obtain that~$\lift_{\bbmB_S, \{x\}}(\calC') \in 
\bullet(\{ \frakI_{\bbmB_S}(Y)  \mid \{x\}
  \subseteq Y \subsetneq S\})$, hence in particular
$\lift_{\bbmB_S, \{x\}}(\calC') \in 
\bullet(\{ \frakI_{\bbmB_S}(Y)  \mid \emptyset
  \subsetneq Y \subseteq S\})$.
  But notice that~$\lift_{\bbmB_S, \{x\}}(\calC')$ is actually
  $\bbmB_S \setminus \calP$, which concludes the proof.
\end{proof}

\subsection{Proof of Theorem~\ref{thm:partial}}
  \label{subsec:proof}

  We now prove Theorem~\ref{thm:partial}. 
  Fix the finite set~$S$ and the configuration~$\calI$ that is a downset of~$\bbmB_S$ such that $\ntz_{\bbmB_S}(\calI)$ is not empty, and let~$Z\in \ntz_{\bbmB_S}(\calI)$.
  By Proposition~\ref{prp:mob-eul} we have
  $\muhat_{\bbmB_S,\calI}(Z) = (-1)^{|Z|} \times \eul(\calI \cap
  {\frakF_{\bbmB_S}(Z)})$,
  therefore $\eul(\calI \cap {\frakF_{\bbmB_S}(Z)}) = 0$.
  We now instantiate Proposition~\ref{prp:eul-equiv} with~$\calG = \calC_1 =
  \calI \cap {\frakF_{\bbmB_S}(Z)}$
  and~$\calC_2 = \emptyset$ to obtain~$\calI \cap {\frakF_{\bbmB_S}(Z)}
  \simeq_{\ap_{\bbmB_S}(\calI \cap {\frakF_{\bbmB_S}(Z))}} \emptyset$.
  This gives us a left-linear tree~$T$ for $\calI \cap
  {\frakF_{\bbmB_S}(Z)}$ with disjoint unions and
subset complements and whose leaves are adjacent pairs from $\ap_{\bbmB_S}(\calI \cap {\frakF_{\bbmB_S}(Z)})$. Next we use 
  Lemma~\ref{lem:pairs-to-PIs} to simulate, in~$\frakF_{\bbmB_S}(Z)$, each
  such adjacent pair using the principal downsets of $\frakF_{\bbmB_S}(Z)$,
  except the one generated by~$Z$ (thanks to the fact that $\emptyset
  \subsetneq Y$ in the statement of the lemma). Therefore, we obtain that
  $\calI \cap {\frakF_{\bbmB_S}(Z)} \in \bullet(\{\frakI_{\bbmB_S}(X) \cap
  {\frakF_{\bbmB_S}(Z)}  \mid X \in \calI, Z \subsetneq X\})$.
  We now use Lemma~\ref{lem:lifting} to obtain that $\lift_{\bbmB_S,Z}(\calI
  \cap {\frakF_{\bbmB_S}(Z)}) \in \bullet(\{\downarrow_{\bbmB_S}(X) \mid X\in
  I\setminus \{Z\}\})$. 
  Notice that $\lift_{\bbmB_S,Z}(\calI
  \cap {\frakF_{\bbmB_S}(Z)}) \cap {\frakF_{\bbmB_S}(Z)} = \calI \cap
  {\frakF_{\bbmB_S}(Z)}$, i.e., restricted to~$\frakF_{\bbmB_S}(Z)$ 
  we have the correct configuration, and we now only need to “fix” what is outside.

  Now let~$\calC'_1 = \lift_{\bbmB_S,Z}(\calI
  \cap {\frakF_{\bbmB_S}(Z))} \setminus (\calI \cap {\frakF_{\bbmB_S}(Z)})$ and
  $\calC'_2 = \calI \setminus (\calI \cap {\frakF_{\bbmB_S}(Z)})$.
  We use Lemma~\ref{lem:allreach} to obtain that~$\calC_1'$ is in 
$\bullet(\{\downarrow_{\bbmB_S}(X) \mid X\in \calC_1'\})$, and likewise
for~$\calC_2'$. Observe that $\calC_1'$ and $\calC_2'$ only consist of elements
outside of $\frakF_{\bbmB_S}(Z))$, in particular $\calC_1', \calC_2' \subseteq I
\setminus \{Z\}$, so that
$\calC_1'$ and $\calC_2'$ are in
  $\bullet(\{\downarrow_{\bbmB_S}(X) \mid X\in
  I\setminus \{Z\}\})$. Finally we can obtain~$\calI$ with the expression~$(\lift_{\bbmB_S,Z}(\calI
  \cap {\frakF_{\bbmB_S}(Z)}) \minusdot \calC_1') \cupdot \calC_2'$, which is in
  $\bullet(\{\downarrow_{\bbmB_S}(X) \mid X\in I\setminus \{Z\}\})$ by what
  precedes. This concludes the proof.

%% file: figures/chainswap.tex
\begin{tikzpicture}	
\tikzset{nodestyle/.style={draw,rectangle}}
\foreach \i [evaluate=\i as \ii using \i*1.5] in {0,1,2,3,4} {
\foreach \j in {0,1,2,3,4} {
	\node[name=n\i\j,nodestyle] at (\j, \ii) {$X_\j$};
		}
	\foreach \k/\l in {0/1,1/2,2/3,3/4} {
			\draw[black, thick] (n\i\k) -- (n\i\l);
			}
			}

\foreach \i [evaluate=\i as \ii using \i*1.5] in {0,1,2,3} {
	\ifodd\i
	\node at (2,\ii+0.75) {$\rewr{$+ \ap_{\bbmB_S}(P)$}{8}$};
	\else
	\node at (2,\ii+0.75) {$\rewr{$- \ap_{\bbmB_S}(P)$}{8}$};
	\fi
	}

	\node[nodestyle,fill=orange] at (0, 4*1.5) {$X_0$};

	\node[nodestyle,fill=orange] at (0, 3*1.5) {$X_0$};
	\node[nodestyle,fill=orange] at (1, 3*1.5) {$X_1$};
	\node[nodestyle,fill=orange] at (2, 3*1.5) {$X_2$};

	\node[nodestyle,fill=orange] at (2, 2*1.5) {$X_2$};

	\node[nodestyle,fill=orange] at (2, 1*1.5) {$X_2$};
	\node[nodestyle,fill=orange] at (3, 1*1.5) {$X_3$};
	\node[nodestyle,fill=orange] at (4, 1*1.5) {$X_4$};

	\node[nodestyle,fill=orange] at (4, 0*1.5) {$X_4$};
\end{tikzpicture}

%% file: final.tex
In this section we present a generalization of Theorem~\ref{thm:partial} that
we believe to be true, present variants of the conjectures, and talk about our
experimental search for counterexamples. 

\paragraph*{Extension of Theorem~\ref{thm:partial}.}
We sketch here a proposed generalization that allows us to avoid not one single zero,
but subsets of zeros while requiring that the lattice has a certain
structure. We first define a few notions.

\begin{definition}
  \label{def:covering-zeros}
  Let~$\calI$ be a downset of~$\bbmB_S$.
  For~$X\in \bbmB_S$, the \emph{non-trivial covering zeros of~$X$ relative
  to~$\calI$} are
  \[\ntcz_{\bbmB_S,\calI}(X) \defeq \{Z \in \ntz_{\bbmB_S}(\calI) \mid  X \subsetneq Z \text{ and there is no } Z' \in \ntz_{\bbmB_S}(\calI) \text{ s.t. } X \subsetneq Z' \subsetneq Z\}. \]
\end{definition}

In other words, $\ntcz_{\bbmB_S,\calI}(X)$ simply consists of the minimal
non-trivial zeros of $\calI$ that are strictly above~$X$.
Observe that we have~$\ntcz_{\bbmB_S,\calI}(X) = \emptyset$ whenever~$X\notin
\calI$. We now define \emph{$k$-decomposable} downsets.

\begin{definition}
  \label{def:2decomp}
  A configuration~$\calI$ of~$\bbmB_S$ that is a downset is
  called \emph{$k$-decomposable} if for every~$X\in \bbmB_S$ we have
  $|\ntcz_{\bbmB_S,\calI}(X)| \leq k$.
\end{definition}

This intuitively means that, in every upset, there are at most two non-trivial
covering zeroes.
We believe that the following is true, but have not completely formalized the
proof.

\begin{conjecture}
  \label{conj:main}
For every finite~$S$, for every configuration~$\calI$
of~$\bbmB_{S}$ that is a $2$-decomposable downset of~$\bbmB_{S}$, we
have that~$\calI \in \bullet(\ncpd_{\bbmB_{S}}(\calI))$.
\end{conjecture}

\begin{proof}[Proof sketch.]
  For an element~$X \in \bbmB_S$, let~$I_X \defeq \{{\frakI_{\bbmB_S}(Y)} \cap
  {\frakF_{\bbmB_S}(X)} \mid Y \in {\frakF_{\bbmB_S}(X)}, \allowbreak \muhat_{\bbmB_S,\calI}(Y)
\neq 0 \}$. The idea of the proof would be to show by top-down induction on~$\bbmB_S$
  that for every~$X \in \bbmB_S$, we have~$\calI \cap {\frakF_{\bbmB_S}(X) \in
  \bullet(I_X)}$. Applying this claim to~$X = \emptyset$ yields the desired
result. To show this, we use similar tools to those developed for the proof
of Theorem~\ref{thm:partial}, extended with an accounting of Euler
  characteristics of diverse sets of configurations. It does not seem that the
  proposed proof would directly generalize to the general case, or indeed to
  $3$-decomposable configurations.
\end{proof}

\paragraph*{Strengthenings.} We present here two possible orthogonal ways in which we can\linebreak
strengthen the conjecture; we do not know whether the stronger conjectures are
true, nor whether they are equivalent to the original conjecture.

First, in view of Proposition~\ref{prp:full-multiplicity}, we could require that
the non-trivial intersections are only used positively or only negatively depending
on the sign of their Möbius value. Indeed, 
Proposition~\ref{prp:full-multiplicity} implies that, on tight intersection
lattices, the non-trivial intersections have total multiplicity equal to their
Möbius value, but they may be used both positively and negatively. Likewise, in
view of Proposition~\ref{prp:ncpdfull}, we could require that the principal
downsets are only used positively and negatively. 

Second, we do not know if the witnessing trees can be required to be
left-linear, i.e., whether we can obtain $\onehat$ (for the $\nci$ conjecture)
or the desired configuration (for the $\ncpd$ conjecture) by a series of
operations where, at each step, we add a (disjoint) non-trivial intersection (or principal
downset), or subtract one (which must be a subset). For instance, this
formulation does not allow us to express a disjoint union of two configurations
themselves obtained with more complex witnessing trees. This stronger conjecture
is the topic of the question asked in~\cite{cstheory} (up to replacing upsets by
downsets and to working in general lattices instead of only on Boolean lattices). This question also explains why the conjecture is false if asked
about general DAGs (instead of lattices), and shows that the conjecture can be
true for so-called \emph{crown-free lattices}.

We note that the construction for the proof of Theorem~\ref{thm:partial} (or
the one that we have in mind for Conjecture~\ref{conj:main}) does not satisfy either of
these strengthenings. However, the proof of Fact~\ref{fact:annoyingfact} can be
adjusted to satisfy the left-linear condition.

\paragraph*{Counterexample search.}
We have implemented a search for counterexamples of the $\nci$ conjecture. We
consider Sperner families~\cite{sperner}, which give all sets of subsets of a
base set of $n$ elements that are non-equivalent (i.e., that are not the same up
to permuting the elements). We have checked in a brute-force fashion that the conjecture holds up to
$n=5$, i.e., on all intersection lattices such that $\onehat$ has cardinality at
most $5$. Unfortunately, already for $n=6$ there are intersection lattices that
are too large for the computation to finish sufficiently quickly. We also
generated random intersection lattices over larger sets of elements, but could
not find a counterexample (assuming that our code is correct). The code checks
the strong version of the $\nci$ conjecture in the previous terminology, i.e.,
it searches for left-linear trees and requires the non-trivial intersections to
be used precisely with the right polarity. The code is available
as-is~\cite{code}. We have also improved the implementation to search for
solutions on each lattice using a SAT-solver rather than a brute-force search:
this speeds up the computation but also did not yield any counterexample.

We had also implemented, earlier, a search for counterexamples of an
alternative phrasing of the $\ncpd$ conjecture: see~\cite{note}. This also
illustrates that the $\ncpd$ conjecture is false if the configuration to reach
is not a downset; see Figure~4 of~\cite{note} (up to reversing directions, i.e.,
considering downsets instead of upsets).